\documentclass[draft]{amsproc}
\usepackage{amsmath}
\usepackage{amsfonts}
\usepackage[dvips]{graphicx}
\usepackage{amssymb}
\usepackage{amsbsy}
\usepackage{amsthm}
\usepackage{graphics}
\usepackage{color}
\usepackage{textcomp}
\usepackage{tikz}
\usetikzlibrary{arrows.meta}
\usepackage{ mathrsfs }

\makeatletter
\@namedef{subjclassname@2010}{%
\textup{2010} Mathematics Subject Classification}
\makeatother

\numberwithin{equation}{section}

\newtheorem{theorem}{Theorem}[section]

\newtheorem{lemma}[theorem]{Lemma}

\newcommand{\ttt}{\mathscr{T}}
\newcommand{\ann}{\mathbb{A}}

\newcommand{\scd}{{T^{\prime*}}}
\newcommand{\cd}{T^{\prime}}

\newcommand{\hil}{\mathcal{H}}
\newcommand{\chil}{\mathscr{H}}

\newcommand{\ddd}{\mathcal{D}}
\newcommand{\natu}{\mathbb{N}}

\newcommand{\bou}{\boldsymbol{B}(\hil)}

\newcommand{\sbou}{\boldsymbol {B}}
\newcommand{\boue}{\boldsymbol {B}(E)}
\newcommand{\spec}{\sigma(T)}

\newcommand{\rad}{r(T)}
\newcommand{\cal}{\mathbb{Z}}

\newcommand{\nul}{\mathcal{N}(}
\newcommand{\comp}{\mathbb{C}}

\newcommand{\disc}{\mathbb{D}}
\newcommand{\elu}{\ell^2(X)}
\newcommand{\pe}{P_E}
\newcommand{\mul}{\mathscr{M}_z}
\newcommand{\com}{C_{\varphi,w}}
\newcommand{\cdcom}{C_{\phi^\prime,w}}

\newcommand{\jad}{\kappa_\chil}

\theoremstyle{definition}
\newtheorem{ex}[theorem]{Example}

\newcommand*{\Le}{\leqslant}
\newcommand*{\Ge}{\geqslant}

\newcommand{\la}{\langle}
\newcommand{\ra}{\rangle}
\newcommand{\gwon}{[E]_{T^*,\cd}}

\DeclareMathOperator{\gen}{Gen}
\DeclareMathOperator{\lin}{lin}
\DeclareMathOperator{\card}{card}
\DeclareMathOperator{\czil}{Chi}

\DeclareMathOperator{\des}{Des}

\DeclareMathOperator{\parr}{par}

\DeclareMathOperator{\hol}{Hol}

\newcommand{\cycle}{\mathscr{C}_{\varphi}}

\newcommand{\set}[1]{\left\{#1\right\}}

\newcommand{\bspace}{\mathcal{B}}
\newcommand{\dualunit}{\mathcal{U}}
\newcommand{\fjad}{\Psi}

\newcommand{\lein}{\mathscr{L}}
\newcommand*{\D}{\mathrm{d\hspace{.1ex}}}

\begin{document}

\title[On Cauchy dual  and  duality for Banach spaces of analytic functions]{On Cauchy dual operator and  duality for Banach\\ spaces of analytic functions }
   \author[P. Pietrzycki]{Pawe{\l} Pietrzycki}
   \address{Wydzia{\l} Matematyki i Informatyki, Uniwersytet
Jagiello\'{n}ski, ul. {\L}ojasiewicza 6, PL-30348
Krak\'{o}w}
   \email{pawel.pietrzycki@im.uj.edu.pl}
   \begin{abstract}
   In this paper, two related types of dualities are investigated. The first is the duality between left-invertible operators and the second is the duality between Banach  spaces  of  vector-valued  analytic  functions. We will examine  a pair ($\mathcal{B},\Psi)$ consisting of a reflexive Banach spaces $\mathcal{B}$ of vector-valued analytic functions on which a left-invertible multiplication operator acts and  an  operator-valued holomorphic function  $\Psi$ on an open subset of complex plane $\Omega$. We prove that there exist a dual pair ($\mathcal{B}^\prime,\Psi^\prime)$ such that the space $\mathcal{B}^\prime$ is unitarily equivalent to the space $\mathcal{B}^*$ and 
  the following intertwining relations hold
  \begin{equation*}
      \mathscr{L} \mathcal{U} = \mathcal{U}\mathscr{M}_z^* \quad\text{and}\quad
    \mathscr{M}_z\mathcal{U} = \mathcal{U} \mathscr{L}^*,
  \end{equation*}
where $\mathcal{U}$ is the unitary operator between $\mathcal{B}^\prime$ and $\mathcal{B}^*$. In addition we show that $\Psi$ and  $\Psi^\prime$ are connected through the relation
\begin{align*}
    \la(\Psi^\prime( \bar{z}) e_1) (\lambda),e_2\ra=
  \la e_1,(\Psi( \bar{ \lambda}) e_2)(z)\ra
\end{align*}
for every $e_1,e_2\in E$, $z\in \varOmega$, $\lambda\in \varOmega^\prime$.

If a left-invertible operator $T\in \bou$
satisfies certain conditions, then both  $T$ and the  Cauchy dual operator $T^\prime$ can be modelled as a multiplication operator on reproducing
kernel Hilbert  spaces of vector-valued analytic functions $\mathscr{H}$ and $\mathscr{H}^\prime$, respectively.
We prove that Hilbert space of the dual pair of $(\mathscr{H},\Psi)$ coincide with $\mathscr{H}^\prime$, where $\Psi$ is a certain operator-valued holomorphic function. Moreover, we characterize when the duality between spaces $\mathscr{H}$ and $\mathscr{H}^\prime$  obtained by identifying
them with $\hil$ is the same as the duality obtained from the Cauchy pairing.

\end{abstract}
   \maketitle
   
   \section{Introduction}
   
Duality is one of the mathematical principles that allows one to look at the same object from two points of view. This is its great advantage and one of the reasons why it attracts the attention of researchers. Different types of dualities appear in many branches of mathematics and physics. We refer the reader to a nice
survey article by M. Atiyah \cite{ati} concerning
this topic. In this paper, we consider two related types of dualities. The first is the duality between left-invertible operators and the second is the duality between Banach  spaces  of  vector-valued  analytic  functions.




Let $\bspace$ be a Banach space of analytic functions on the unit disc $\disc$ continuously contained in the space $\hol(\disc)$ with the topology given by uniform convergence on compact sets. Assume that  $\bspace$ contains the space $\hol(\overline{\disc})$ of analytic functions in a neighbourhood of
$\disc$ as a dense subset. Define an operator $\dualunit:\bspace^*\to \hol(\disc)$ by

\begin{equation}\label{konstr}
    \overline{\dualunit\varphi(\lambda)}:=\varphi\big(\frac{1}{1-\bar{\lambda}z}\big),\qquad  \varphi \in \bspace^*.
\end{equation}
As shown in \cite[p. 616]{pred}  $\dualunit$ is injective. Now let $\bspace^\prime$ be the image of $\bspace^*$ by $\dualunit$ with the norm induced from $\bspace^*$, so that $\dualunit$ is unitary. With this norm the space $\bspace^\prime$ is called
the Cauchy dual of $\bspace$. 
The name is justified by the fact that the dual of $\bspace$ is represented by $\bspace^\prime$  via the Cauchy pairing
\begin{equation}\label{cauparfi}
    \varphi(f)=\lim_{r\to1^-}\int_0^{2\pi}f(re^{it})\dualunit\varphi(re^{it})\frac{dt}{2\pi}, \qquad f\in \hol(\overline{\disc}), \varphi \in \bspace^*.
\end{equation}
In Banach spaces $\bspace$ of analytic functions where the dilations\footnote{The dilation $f_r:\disc\to \comp$ is defined by $f_r(z):=f(rz)$, $z\in \comp$.}
$f_r$ , $0 < r < 1$ of a function $f\in \bspace$ converge to $f$ in $\bspace$ as $r \to 1^+$ the above holds for all $f\in\bspace$ (see \cite[p. 616]{pred}).
This notion is
well-known in the theory of spaces of analytic functions. For example, the dual of Bergman space $\bspace$ of analytic functions on disc can be identified of course as $\bspace$
itself, via the usual Hilbert-space duality, defined using the  inner
product. However, in many applications it is more appropriate to  identify the dual of $\bspace$ with the Dirichlet
space of analytic functions on disc via the Cauchy pairing (see \cite[Example 1.4]{ale}).
We note that the notion of Cauchy duality is close to the notion of triplet
of Hilbert spaces when the middle space is the Hardy space $H^2$, which is why it is also called the $H^2$-duality.
For more examples and information on Cauchy duals we
refer the reader to \cite{ale,shap,pred}.





 In \cite{shi} S. Shimorin
 constructed an analytic model for a left-invertible analytic
 operator $T\in \bou$. Namely, he showed that operator $T$ is unitarily equivalent to multiplication operator acting in some reproducing kernel Hilbert space $\chil$ of vector-valued
 holomorphic functions  defined on a disc. The construction of
 this analytic model  is based on the following  unitary isomorphism:
 \begin{equation*}
   U: \hil \ni x \to
\sum_{n=0}^\infty
(P_E{T^{\prime*n}}x)z^n\in \chil,\quad z\in \disc({r(\cd)}^{-1}),
 \end{equation*}
where $E:=\nul T^*)$ and $\cd$ is 
 the Cauchy dual operator  of $T$.
 The Cauchy dual operator $\cd$ of a left-invertible $T\in\bou$ is  defined by
\begin{equation*}
 \cd:=T(T^*T)^{-1}
\end{equation*}
and was introduced and studied by S. Shimorin in \cite{shi}.
The Cauchy dual operator of a left-invertible analytic operator $T$ is itself left-invertible.
Moreover, if  $\cd$ is also analytic,
then for both operators $T$ and $\cd$ one can construct Hilbert spaces $\chil$ and $\chil^\prime$ of vector-valued
 holomorphic functions defined on a disc. S. Shimorin observed  that the duality between $\chil$ and $\chil^\prime$ obtained by identifying
them with $\hil$ is the same as the duality obtained from the Cauchy pairing, that is, 

\begin{equation}\label{shimcaupar}
    \la U^{-1}f,U^{\prime-1}g\ra_\hil=\sum_{n=0}^{\infty}\la \hat{f}(n),\hat{g}(n)\ra_E
\end{equation}
for $E$-valued polynomials
\begin{equation*}
    f(z)=\sum_{n=0}^{\infty}\hat{f}(n)z^n\in \chil \quad\textrm{and}\quad g(z)=\sum_{n=0}^{\infty}\hat{g}(n)z^n \in \chil^\prime.
\end{equation*}

In the recent paper \cite[Section 3]{ja3}, the author provided a new  analytic model on an annulus for left-invertible operators, which are not necessarily analytic operators.
The construction of
 the analytic model on an annulus for a left-invertible  operator $T\in \bou$ is based on the following  unitary isomorphism:
  \begin{equation*}
   U:\hil \ni x\rightarrow 
\sum_{n=1}^\infty
(P_ET^{n}x)\frac{1}{z^n}+
\sum_{n=0}^\infty
(P_ET^{\prime*n}x)z^n \in \chil,
 \end{equation*}
where $E$ is a closed subspace of $\hil$ satisfying some certain condition (see \eqref{li}).
The model extends both Shimorin's analytic model for left-invertible analytic operators (see \cite[Theorem 3.3]{ja3}) and Gellar's model for a bilateral weighted shift (see \cite[Example 5.2]{ja3}). As shown in \cite[Theorem 3.2 and 3.8]{ja3} a left-invertible operator $T$, which satisfies certain conditions can be modelled as a multiplication operator $\mathscr{M}_z$ on a reproducing
kernel Hilbert space of vector-valued analytic functions on an annulus or a disc. We refer the reader to \cite{shi,ja3} for more information on analytic model for left-invertible operator. 
For other results  related to the analytic model see,
for example \cite{chav,dym,ja4,Sark}.

It is worth noting that the notion of the Cauchy dual operator is also interesting because  the map $T\to T^\prime$ sets up  the correspondence between some classes of operators:
\begin{center}
\begin{tabular}{rrl} 
$T$ &$\leadsto$& $T^\prime$  \\
 expansion   &  $\leadsto$ &  contraction \\ 
 2-hyperexpansive operator &  $\leadsto$ &hyponormal contraction\\
  completely hyperexpansive & $\leadsto$ & contractive
subnormal \\
 weighted shift & & weighted shift\\
 $\cdots$ & $\leadsto$ & $\cdots$\\
\end{tabular}
\end{center}
(see \cite[pp. 639/640]{2hyp}).  
Recently the Cauchy dual subnormality problem, which asks whether the Cauchy
dual operator of a 2-isometry is subnormal was solved negatively in
the class of 2-isometric operators (see \cite{ana}).
The topics related to the Cauchy dual operator are
currently being studied intensively from several points of view (see e.g.  \cite{ana,ana2,badea,2hyp,un2hyp,curto,ezz}).







Let $T$ be a left-invertible operator such that both operators $T$ and $T^\prime$ can be modelled as a multiplication operator on a reproducing
kernel Hilbert space of vector-valued analytic functions on an annulus.
In view of the above, one may  ask whether the duality between $\chil$ and $\chil^\prime$ obtained by identifying
them with $\hil$ is the same as the duality obtained from the Cauchy pairing? Can the Hilbert space $\chil^\prime$ be constructed in a similar way to \eqref{konstr}?

Research on these problems led us to  investigate  a pair ($\bspace,\fjad)$ consisting of  a reflexive Banach spaces $\bspace$ of vector-valued analytic functions on which a left-invertible multiplication operator $\mul$ acts and  an  operator-valued holomorphic function  $\fjad:(\varOmega^\prime)^*\rightarrow \sbou(E,\bspace)$ which satisfies the following conditions:
\begin{itemize}
    \item[(A1)]  $\bspace \hookrightarrow \hol(\varOmega, E)$ the inclusion map is both injective and
continuous
(the space $\hol(\varOmega, E)$ with the topology of uniform convergence on compact sets),
  
    \item[(A2)]
    the subspace $\lin \{\fjad (\bar{\lambda})e \colon e\in E, \: \lambda \in \varOmega^\prime \}$ is dense in $\bspace$,
     
    \item[(A3)] $\mathscr{L} \fjad( \bar{\lambda}) = \bar{\lambda}\fjad ( \bar{\lambda})$ for every $\lambda \in \varOmega^\prime$,
\end{itemize}
where $E$ is a Hilbert space and $\varOmega$, $\varOmega^\prime\subset \comp$ are open sets.

It turns out that  this type of Banach spaces include the classical Banach spaces of holomorphic functions in the unit disc: the Hardy space, the  Bergman space and the 
Dirichlet space (see Example \ref{aleman}) as well as the Hilbert spaces  of vector-valued analytic functions on an annulus $\chil$ associated with the analytic models for  left-invertible operators (see Examples \ref{shmod} and \ref{ppmod}).
We prove that there exist a dual pair ($\bspace^\prime,\fjad^\prime)$ such that the space $\bspace^\prime$ is unitarily equivalent to the space $\bspace^*$ and 
  the following intertwining relations hold
  \begin{equation*}
    \lein \dualunit = \dualunit\mul^*\quad\text{and}\quad
    \mul\dualunit = \dualunit \lein^*,
\end{equation*}
where $\dualunit$ is the unitary operator between $\bspace^\prime$ and $\bspace^*$. In addition we show that $\fjad$ and  $\fjad^\prime$ are connected through the relation
    \begin{equation*}
    \la(\fjad^\prime( \bar{z}) e_1) (\lambda),e_2\ra=
  \la e_1,(\fjad( \bar{ \lambda}) e_2)(z)\ra
\end{equation*}
for every $e_1,e_2\in E$, $z\in \varOmega$, $\lambda\in \varOmega^\prime$.
We define the dual space $\bspace^\prime$ as the image of $\bspace^*$ by $\dualunit:\bspace^*\to \hol(\varOmega^\prime, E)$ with the norm induced from $\bspace^*$, where $\dualunit$  is given by \begin{equation*}\label{1uo}
    \varphi(\fjad  (\bar{\lambda})e) = \la e, (\dualunit\varphi)(\lambda) \ra, \qquad e \in E, \lambda \in \varOmega^\prime.
\end{equation*}

We describe the relationship between the analytic model for $T$ and analytic model for the Cauchy dual operator $T^\prime$. Namely, we prove that Hilbert space associated with the analytic model for $T^\prime$ coincides with the Hilbert space  obtained in the above construction with $\chil$ in place of $\bspace$ and $\fjad:(\varOmega^\prime)^*\rightarrow \sbou(E,\chil)$
defined by
\begin{equation*}
 \fjad (\lambda):=  \sum_{n=1}^\infty \frac{1}{\lambda^n}\mathscr{L}^n +  \sum_{n=0}^\infty\lambda^n \mul^n,\qquad \lambda \in(\varOmega^\prime)^* .
\end{equation*}
Moreover, we characterize left-invertible operators, for which the duality between $\chil$ and $\chil^\prime$ obtained by identifying
them with $\hil$ is the same as the duality obtained from the Cauchy pairing.



   \section{Preliminaries}\label{prel}
   In this paper, we use the following notation. The
field of complex
numbers is denoted by
$\mathbb{C}$. The
symbols $\mathbb{Z}$, $\mathbb{Z}_{+}$ and $\mathbb{N}$
 stand for the sets of integers,
positive integers and nonnegative integers, respectively. Set 
$\disc(r)=\set{z\in \comp\colon |z|< r}$ and
$\ann(r^-,r^+)=\set{z\in \comp\colon r^-< |z|< r^+}$ for $r,r^-,r^+\in[0,\infty)$. If $\varOmega\subset\comp$   then $\varOmega^*$ is the set $\{\bar{z}\colon z\in\varOmega\}$. We denote by $\card(X)$ the cardinal number of a
set $X$. 

All Hilbert spaces considered in this paper are assumed to be complex. Let $T$ be a linear operator in a complex Hilbert
space $\mathcal{H}$. Denote by  $T^*$  the adjoint
of $T$.  We write $\bou$ 
 for the $C^*$-algebra of all bounded operators in $\mathcal{H}$. Let $T \in\bou$.
The spectrum and spectral radius of $T\in\bou$ is denoted by $\spec$ and $\rad$
respectively.

We say that $T$ is \text{left-invertible} if there exists $S \in \bou$ such
that $ST = I$. 
We call $T$ \textit{analytic} if $\hil_\infty=\bigcap_{i=1}^\infty T^i\hil=\set{0}$.

Let $X$ be a countable set and $\varphi:X\to X$ be a selfmap. If $n\in \mathbb{Z}_+$, then the $n$-th iterate of $\varphi$ is given by $\varphi^{(n)}=\underbrace{\varphi\circ\varphi\circ\dots\circ\varphi}_n$, $\varphi$ composed with itself $n$-times and $\varphi^{(0)}$ is identity function. For $x\in X$ the set 
   \begin{equation*}
 [x]_\varphi=\{y\in X: \text{there exist } i,j\in \natu \text{  such that  } \varphi^{(i)}(x)=\varphi^{(j)}(y) \}
   \end{equation*} 
   is
called the \textit{orbit of}  $\varphi$ containing $x$. If $x\in X$ and $\varphi^{(i)}(x)=x$ for some $i\in \mathbb{Z}_+$, then the \textit{cycle of} $\varphi$ containing $x$ is the set
\begin{equation*}
  \cycle=\{\varphi^{(i)}(x)\colon i\in \natu \}.
\end{equation*}
   Define the function $[\varphi]:X\rightarrow \mathbb{Z}$ by
   \begin{itemize}
       \item[(i)] $[\varphi](x)=0$ if  $x$ is in the cycle of $\varphi$
       \item[(ii)] $[\varphi](x^*)=0$, where $x^*$ is a fixed element of orbit $F$ of $\varphi$ not containing a cycle,
       \item[(iii)]
     $[\varphi](\varphi(x))=[\varphi](x)-1$ if $x$ is not in a cycle of $\varphi$.
   \end{itemize}We set
    \begin{equation*}
    \gen_\varphi{(m,n)}:=\{x\in X\colon m\Le[\varphi](x)\Le n\}
\end{equation*}for $m,n\in\mathbb{Z}$.
    We say that $\varphi$ has \textit{finite branching index} if \begin{equation*}\sup \set{|[\varphi](x)|: \card(\varphi^{-1}(x))\Ge2,\:\: x\in X}< \infty.
 \end{equation*}
 Let $w:X \rightarrow \comp$ be
a complex function on $X$.
By a \textit{weighted composition 
operator} $\com$ in $\elu$ we mean a mapping
\begin{align*}
\ddd(\com)&=\{f\in \elu : w(f\circ\varphi)\in \elu\},
 \\\notag
\com f&=w(f\circ\varphi),\quad f \in\ddd(\com).
\end{align*} 
We call $\varphi$ and $w$ the \textit{symbol} and the \textit{weight} of $\com$ respectively. 
Let us recall some useful properties of composition operator we need
in this paper:

\begin{lemma}[{\cite[Lemma 2.1]{ja3}}]\label{podst}Let $X$ be a countable set, $\varphi: X \rightarrow X$ be a selfmap and $w:X \rightarrow \comp$ be a complex function. If $\com\in \sbou(\ell^2(X))$, then for any $x\in X$ and $n\in \mathbb{Z}_+$
\begin{itemize}
\item[(i)]$\com^*e_x=\overline{w(x)}e_{\varphi(x)}$,
\item[(ii)] $\com e_x=\sum_{y\in\varphi^{-1}(x)}w(y)e_y$,
 \item[(iii)]$\com^{*n} e_x=\overline{w(x)w(\varphi(x))\cdots w(\varphi^{(n-1)}(x))}e_{\varphi^{(n)}(x)}$,
       \item[(iv)]$\com^n e_x=\sum_{y\in\varphi^{-n}(x)}w(y)w(\varphi(y))\cdots w(\varphi^{(n-1)}(y))e_y$,
     \item[(v)] $\com^*\com e_x=\Big(\sum_{y\in\varphi^{-1}(x)}|w(y)|^2\Big)e_x$.
   \end{itemize}
  
   \end{lemma}
 
    We now describe the Cauchy dual of a weighted composition operator.
    
   \begin{lemma}[{\cite[Lemma 2.2]{ja3}}]\label{cdcom}Let $X$ be a countable set, $\varphi: X \rightarrow X$ be a selfmap and $w:X \rightarrow \comp$ be
a complex function. If $\com\in \sbou(\ell^2(X))$ is a left-invertible operator, then the Cauchy dual $\com^\prime$ of $\com$ is also a weighted composition operator $\cdcom$ with the same symbol $\varphi:X\to X$ and weight $w^\prime:X\to\comp$ defined by
\begin{equation*}
  w^\prime(x):=  \frac{w(x)}{\Big(\sum_{y\in\varphi^{-1}(\varphi(x))}|w(y)|^2\Big)}.
\end{equation*}
   \end{lemma}

\section{ Analytic model}\label{anamodel}

Since  the  analytic  model for left-invertible operator introduced in the recent paper \cite{ja3} by the author plays  a major role in this paper, we outline it in the following discussion.   Let $T\in \bou$  be a left-invertible operator and $E$ be a closed subspace of $\hil$ denote by $\gwon$ 
the  following subspace of $\hil$:
\begin{equation*}
\gwon:=\bigvee\big(\{T^{*n}x\colon x\in E, n\in \natu\}\cup\{T^{\prime n} x\colon x\in E, n\in \natu\}\big),
\end{equation*}
where $\cd$ is the Cauchy dual of $T$. 

To avoid repetition, we state the following assumption which will be used
frequently in this paper.
   \begin{align} \tag{$\clubsuit$} \label{li}
\begin{minipage}{70ex} 
The operator  $T\in \bou$ is left-invertible   and $E$ is a closed subspace of $\hil$ such that $\gwon=\hil$.
\end{minipage}
    \end{align}
  Suppose \eqref{li} holds. In this case we may construct a Hilbert $\chil$
 associated with $T$, of formal Laurent series with vector  coefficients. We proceed as follows. For
each $x \in\hil$, define a formal Laurent series $U_x$ with vector  coefficients  as  
   \begin{equation}\label{mod}
       U_x(z) =\sum_{n=1}^\infty
(P_ET^{*n}x)\frac{1}{z^n}+
\sum_{n=0}^\infty
(P_E{T^{\prime*n}}x)z^n.
 \end{equation}
   
 Let $\chil$ denote the vector space of formal Laurent series with vector  coefficients of the form $U_x$, $x \in \hil$.  
   Consider the map $U:\hil\rightarrow \chil$ defined by $Ux:=U_x$. 
As shown in \cite[Lemma 3.1]{ja3}  $U$ is injective. In particular, we may equip the space
$\chil$ with the norm induced from $\hil$, so that $U$ is unitary. Observe that every $f\in\chil$ can be represented as follows
 \begin{equation*}
f=\sum_{n=-\infty}^\infty
\hat{f}(n)z^n,
 \end{equation*}
where
\begin{equation}\label{gg}
 \hat{f}(n) = \left\{ \begin{array}{ll}
\pe T^{\prime *n}U^*f & \textrm{if $n\in\natu$},\\
\pe T^{-n}U^*f & \textrm{if $n\in\cal\setminus\natu$.}
\end{array} \right.
 \end{equation}

By \cite[Theorem 3.2]{ja3} the operator $T$ is unitarily equivalent to the operator $\mul:\chil\to\chil$ of multiplication
by $z$ on $\chil$ given by 
\begin{equation*}
    (\mul f)(z)=zf(z),\quad f\in\chil 
\end{equation*}
and operator $\scd$ is unitarily equivalent to the operator $\mathscr{L}:\chil\to\chil$ given by
\begin{equation*}
    (\mathscr{L}f)(z)=\frac{f(z)-(P_{\nul \mul^*)}f)(z)}{z}, \quad f\in\chil.
\end{equation*}
 Following \cite{shi}, the reproducing kernel for $\chil$ is an $\boue$-valued function of two variables $\jad:\varOmega\times\varOmega\rightarrow \boue$ that  \begin{itemize}
       \item[(i)] for any $e\in E$ and $\lambda \in \varOmega$
       \begin{equation*}
           \jad(\cdot,\lambda)e\in \chil,
       \end{equation*}
       \item[(ii)]for any $e\in E$, $f\in \chil$ and $\lambda \in \varOmega$
       \begin{equation*}
           \la f(\lambda),e\ra_E=\la f,\jad(\cdot,\lambda)e\ra_\chil.
       \end{equation*}
   \end{itemize}
It turns out that if the series \eqref{mod} is convergent in $E$ on $\varOmega\subset\comp$ for every $x\in \hil$, then $\chil$ is a
reproducing kernel Hilbert space of vector-valued holomorphic functions on $\varOmega$ (see \cite[Theorem 3.8]{ja3}).

For left-invertible operator $T\in\bou$, among all subspaces satisfying condition \eqref{li} we will distinguish those subspaces $E$ which satisfy the following condition
\begin{equation}\tag{$\spadesuit$}\label{prep}
   E\perp T^n E \qquad\text{and}\qquad E\perp T^{\prime n} E, \qquad n\in \mathbb{Z}_+.
\end{equation}

\section{Duality}\label{dualitysec}
 In this section, we will consider a quintaple ($\bspace$,    $\fjad$, $\lein$, $\mul$) consisting of:   a reflexive Banach space $\bspace$ of $E$-valued analytic functions on which a left-invertible multiplication operator $\mul:\bspace\to\bspace$, defined by
\begin{align*}
    (\mul f)(z)=zf(z), \quad f\in \bspace, 
\end{align*}
acts, $\mathscr{L}:\bspace\rightarrow \bspace$ is a left inverse of $\mul$ and $\fjad:(\varOmega^\prime)^*\rightarrow \sbou(E,\bspace)$ is an  operator-valued holomorphic function, where $E$ is a Hilbert space and $\varOmega$, $\varOmega^\prime\subset \comp$ are open sets. We assume 
  that  the following
conditions hold:
\begin{itemize}
    \item[(A1)]  $\bspace \hookrightarrow \hol(\varOmega, E)$ the inclusion map is both injective and
continuous
(the space $\hol(\varOmega, E)$ with the topology of uniform convergence on compact sets),
  
    \item[(A2)]
    the subspace $\lin \{\fjad (\bar{\lambda})e \colon e\in E, \: \lambda \in \varOmega^\prime \}$ is dense in $\bspace$,
     
    \item[(A3)] $\mathscr{L} \fjad( \bar{\lambda}) = \bar{\lambda}\fjad ( \bar{\lambda})$ for every $\lambda \in \varOmega^\prime$.
\end{itemize}

We prove that there exist a dual quintaple  ($\bspace^\prime$,    $\fjad^\prime$, $\lein$, $\mul$) such that
the space $\bspace^\prime$ is unitarily equivalent to the space $\bspace^*$ and the unitary operator
$\dualunit:\bspace^* \to \bspace^\prime$ between these spaces
intertwines
 $\mul^*$ and
$\lein^*$ on $\bspace^*$ with the  $\lein$ and  $\mul$ on $\bspace^\prime$, respectively, that is:
\begin{equation*}\label{intert}
    \lein \dualunit = \dualunit\mul^*\quad\text{and}\quad
    \mul\dualunit = \dualunit \lein^*.
\end{equation*}
 In addition, we show that $\fjad$ and  $\fjad^\prime$ are connected through the relation
    \begin{equation}\label{fjadrel}
    \la(\fjad^\prime( \bar{z}) e_1) (\lambda),e_2\ra=
  \la e_1,(\fjad( \bar{ \lambda}) e_2)(z)\ra
\end{equation}
for every $e_1,e_2\in E$, $z\in \varOmega$, $\lambda\in \varOmega^\prime$.



        
Now we provide a construction of the dual quintuple
($\bspace^\prime$,    $\fjad^\prime$, $\lein$, $\mul$).
Define an operator $\dualunit:\bspace^*\to \hol(\varOmega^\prime, E)$ such that for every $\varphi \in \bspace^*$ the following equation holds
\begin{equation}\label{cdu}
    \varphi(\fjad(\bar{\lambda})e) = \la e, (\dualunit\varphi)(\lambda) \ra, \qquad e \in E, \lambda \in \varOmega^\prime.
\end{equation}
Now let $\bspace^\prime$ be the image of $\bspace^*$ by $\dualunit$ and equip this space
with the norm induced from $\bspace^*$, so that $\dualunit$ is unitary.
Define an operator  $\lein:=\dualunit \mul^*\dualunit^*$
and  an operator-valued holomorphic  function $\fjad^\prime:\varOmega^*\rightarrow \sbou(E,\bspace^\prime)$ by
\begin{equation}\label{fjadprdef}
    \fjad^\prime (z)e=\dualunit\varphi_{\bar{z},e},\qquad z\in \varOmega^*,\:e\in E,
\end{equation}
where  a functional $\varphi_{z,e} :\bspace \to \comp$ for $e\in E$,  $z\in\varOmega$ is given by
\begin{equation}\label{funct}
    \varphi_{z,e}(f)= \la f(z),e\ra,\qquad f \in \bspace.
\end{equation}
Below, we show that  
 $\dualunit$ and $\fjad^\prime$ are well-defined. 
\begin{lemma}\label{cv} Suppose that the quintuple ($\bspace$,    $\fjad$, $\lein$, $\mul$) is as above.
Then $\dualunit$ and  $\fjad^\prime$ are well-defined. Moreover, the operator $\dualunit$ is  injective.
\end{lemma}
\begin{proof}
First we prove that $\dualunit$ is well-defined. Fix $\varphi\in\bspace^*$. Note that  $E\ni e\to  \varphi(\fjad(\bar{\lambda})e)\in \comp$ for $\lambda \in \varOmega^\prime$ is a continuous linear functional
thus by Riesz-Fr{\'e}chet representation theorem (see \cite[Theorem 12.5]{rud}) and \eqref{cdu} the function $\dualunit\varphi$ is uniquely determined on the set $\varOmega^\prime$. Since $\fjad$ is an operator-valued holomorphic function we infer from \eqref{cdu} that $\dualunit\varphi$ is a weakly holomorphic function. Using the fact that weakly holomorphic functions are strongly holomorphic (see \cite[
Theorem 3.31]{rud}), we see that $\dualunit\varphi$ is strongly holomorphic, which yields 
$\dualunit\varphi\in\hol(\varOmega^\prime, E)$.

Since the set $\lin \{\fjad ( \bar{\lambda})e \colon e\in E, \: \lambda \in \varOmega^\prime \}$ is dense in $\bspace$ the map $\dualunit :\bspace^* \ni \varphi \rightarrow \dualunit \varphi \in \hol(\varOmega^\prime, E)$ is injective.

Now we justify that $\fjad^\prime$ is well-defined. Fix $z\in \varOmega$. 
It is easy to see that $E\ni e\to \varphi_{z,e}\in \bspace^*$ is a linear functional. This and \eqref{fjadprdef} implies that $\fjad^\prime ( z)$ is  linear. By (A1) there exist a constant $C>0$ such that
\begin{equation*}
    \|f(z)\|_E\leq C\|f\|_\bspace, \quad f \in \bspace.
\end{equation*}
This combined with \eqref{funct} shows that
\begin{equation}
   | \varphi_{z,e}(f)|= |\la f(z),e\ra|\Le \|f(z)\|_E\|e\|_E\Le C\|f\|_\bspace\|e\|_E.
\end{equation}
for $f \in \bspace$ and $e\in E$. Hence,  $\varphi_{z,e}$ is bounded and
\begin{equation*}
    \|\varphi_{z,e}\|_{\bspace^*}\Le C\|e\|_E,\qquad e\in E.
\end{equation*}
Thus, by \eqref{fjadprdef}, we see that
\begin{equation}
    \|\fjad^\prime (z)e\|_{\bspace^\prime}=\|\dualunit\varphi_{\bar{z},e}\|_{\bspace^\prime}=\|\varphi_{\bar{z},e}\|_{\bspace^*}\Le C\|e\|_E,
\end{equation}
for all $e\in E$. Therefore $\fjad^\prime (z)$ is a bounded operator for every $z\in \varOmega$.

\end{proof}

 Now we prove that $\fjad^\prime$ is actually an operator-valued holomorphic function. In fact, we give two proofs of this theorem. The first  appeals to  the generalization of Hartogs’  theorem (see \cite[Theorem 36.1]{Muj}) which states that separately vector-valued holomorphic functions are strongly holomorphic. The
second  utilizes  the Cauchy
Integral Formula.

\begin{theorem} \label{holotw} Suppose that the quintuple ($\bspace$,    $\fjad$, $\lein$, $\mul$) is as above. Then the following conditions hold:
\begin{itemize}
    \item[(i)]
    $\fjad^\prime$ is an operator-valued holomorphic function,
    \item[(ii)] $\fjad$ and $\fjad^\prime$ are connected through the relation
    \eqref{fjadrel}
for every $e_1,e_2\in E$, $z\in \varOmega$, $\lambda\in \varOmega^\prime$.
\end{itemize}
\end{theorem}

\begin{proof}

(ii) Combining \eqref{cdu}, \eqref{fjadprdef} with \eqref{funct}, we get
\begin{align*}
\la e_2,(\fjad^\prime ( \bar{z})e_1) (\lambda)\ra&\overset{\eqref{fjadprdef}}=\la e_2, (\dualunit\varphi_{z,e_1})(\lambda) \ra \overset{\eqref{cdu}}{=} \varphi_{z,e_1}(\fjad ( \bar{ \lambda})e_2)\\&\overset{\eqref{funct}}{=}
 \la(\fjad (  \bar{\lambda})e_2)(z),e_1\ra
\end{align*}  
for $z\in \varOmega$, $\lambda \in \varOmega^\prime$, $e_1,e_2\in E$.

(i) Fix $e\in E$. Let $\vartheta \colon\varOmega^*\times \varOmega^\prime\to E$ be a two variable function defined by
\begin{equation*}
    \vartheta(z,\lambda)=(\fjad^\prime ({z})e)  (\lambda),\quad z\in \varOmega^*,\:\lambda\in \varOmega^\prime.
\end{equation*}
Combining \eqref{fjadrel} with the fact that $\fjad$ is an operator-valued holomorphic function and $\bspace$ is a space of holomorphic functions, we deduce that $\vartheta$ is a
separately weakly holomorphic function. More precisely, $\vartheta$ is a weakly holomorphic function in each variable $z$ and $\lambda$, while the other variable is held constant.  Since weakly holomorphic functions are strongly holomorphic (see \cite[Theorem 3.31]{rud}), we deduce that  $\vartheta$ is 
separately strongly holomorphic function. By Hartogs' theorem for vector-valued holomorphic functions (see \cite[Theorem 36.1, p.265]{Muj}), each separately vector-valued holomorphic
function is strongly holomorphic thus  $\vartheta$ is strongly holomorphic.

We will now give an alternative proof of holomorphicity of $\vartheta$ without using the generalized Hartogs' theorem. We claim that $\vartheta$ is a
 continuous function.
Let $z_0\in \varOmega^*$ 
and choose $r>0$ such that
$\overline{\disc(z_0,r)}\subset\varOmega^*$. Put $K:= \overline{\disc(z_0,r)}$.
By (A1)  there exist $C>0$ such that
\begin{equation}\label{isc}
    \|f(z)\|_E\leq C\|f\|_\bspace, \qquad z \in K, f \in \bspace.
\end{equation}
Fix $f\in \bspace$ and $e\in E$. We define a function $g_{f,e}\colon\varOmega\to\comp$ by
\begin{align*}
    g_{f,e}(z) = \la f(z),e\ra,\quad z \in \varOmega.
\end{align*}
Then, by \eqref{isc}
\begin{align*}
    |g_{f,e}(z)|=|\la f({z}),e\ra|\Le \|f({z})\|_E\|e\|_E\Le C \|f\|_\bspace\|e\|_E, \quad z \in K,
\end{align*}
which gives
\begin{align}\label{supest}
   \sup_{z \in K} |g_{f,e}(z)|\Le C \|f\|_\bspace\|e\|_E.
\end{align}
The Cauchy
Integral Formula yields
\begin{align*}
   g_{f,e}(z)-g_{f,e}(z_0)&=\frac{1}{2\pi i}\int_{\partial K}
\Big(\frac{g_{f,e}(\xi)}{\xi-z} -
\frac{g_{f,e}(\xi)}{\xi-z_0}\Big) \D \xi\\&=\frac{z-z_0}{2\pi i}\int_{\partial K}
\frac{g_{f,e}(\xi)}{(\xi-z)(\xi-z_0)} \D \xi
\end{align*}
for $z\in {\disc(z_0,r)}$. By \eqref{supest} and the standard integral estimate
 we obtain,
\begin{align*}
 | g_{f,e}(z)-g_{f,e}(z_0)| &\Le\frac{|z-z_0|}{2\pi}\int_{\partial K}\Big|
\frac{g_{f,e}(\xi)}{(\xi-z)(\xi-z_0)}\Big| \D \xi \Le \frac{2}{r}|z-z_0| \sup_{\xi \in K} |g_{f,e}(\xi)| 
\\&\Le \frac{2}{r}|z-z_0|C \|f\|_\bspace\|e\|_E
 \end{align*}
for $z\in {\disc(z_0,\frac{r}{2})}$. It follows from the above that 
\begin{align*}
      |(\varphi_{{z},e}- \varphi_{{z_0},e})(f)|&\overset{\eqref{fjadprdef}}{=}|\la f({z}),e\ra-\la f({z}_0),e\ra|=| g_{f,e}(z)-g_{f,e}(z_0)|
      \\&\Le  \frac{2}{r}|z-z_0|C \|f\|_\bspace\|e\|_E
\end{align*}
for $z\in {\disc(z_0,\frac{r}{2})}$. Since $f$ was arbitrarily chosen, we get
\begin{align}\label{nazest}
   \|\varphi_{{z},e}- \varphi_{{z_0},e}\|_{\bspace^*}&\Le |z-z_0|C \|e\|_E\frac{2}{r},\quad z\in {\disc(z_0,\frac{r}{2})}.
\end{align}
Let $\lambda_0\in \varOmega^\prime$
and choose $\rho>0$ such that
$\overline{\disc(\lambda_0,\rho)}\subset\varOmega^\prime$. By continuity of $\fjad$, there exist a constant $D>0$ such that $\|\fjad(\lambda)\|_{\sbou(E,\bspace)}<D$ for $\lambda \in\overline{\disc(\lambda_0,\rho)}$. Thus, by \eqref{cdu}
\begin{align*}
  |\la e, (\dualunit\varphi)(\lambda) \ra|\Le   \|\varphi\|_{\bspace^*}\|\fjad(\bar{\lambda})e\|_\bspace \Le D\|\varphi\|_{\bspace^*}\|e\|_E
\end{align*}
 for  $e \in E$, $\lambda \in \overline{\disc(\lambda_0,\rho)}$ and $\varphi\in\bspace^*$, which implies
 \begin{align*}
  \|(\dualunit\varphi)(\lambda) \|_E\Le D\|\varphi\|_{\bspace^*},\quad  \lambda \in \overline{\disc(\lambda_0,\rho)},\: \varphi\in\bspace^*.
\end{align*}
This combined with \eqref{nazest} gives
\begin{align*}
    \|\dualunit\varphi_{\bar{z},e}(\lambda)- \dualunit\varphi_{\bar{z_0},e}(\lambda)\|_E&\Le D\|\dualunit\varphi_{\bar{z},e}- \dualunit\varphi_{\bar{z_0},e}\|_{\bspace^\prime}\Le D\|\varphi_{\bar{z},e}- \varphi_{\bar{z_0},e}\|_{\bspace^*}
    \\&\Le CD|z-z_0| \|e\|_E\frac{2}{r} 
\end{align*}
for $z\in{\disc(z_0,\frac{r}{2})}$ and $\lambda \in \overline{\disc(\lambda_0,\rho)}$. Hence, we have
\begin{align*}
    \|\vartheta(z,\lambda)-\vartheta(z_0,\lambda_0)\|_E &= \|(\fjad^\prime ({z})e)  (\lambda) - (\fjad^\prime ({z_0})e)  (\lambda_0)\|_E
    \\&\hspace{-1ex}\overset{\eqref{fjadprdef}}{=}\|\dualunit\varphi_{\bar{z},e}(\lambda)- \dualunit\varphi_{\bar{z_0},e}(\lambda_0)\|_E\\&\Le \|\dualunit\varphi_{\bar{z},e}(\lambda)- \dualunit\varphi_{\bar{z_0},e}(\lambda)\|_E+\|\dualunit\varphi_{\bar{z_0},e}(\lambda)-
    \dualunit\varphi_{\bar{z_0},e}(\lambda_0)\|_E
    \\&\Le CD|z-z_0| \|e\|_E\frac{2}{r}+\|\dualunit\varphi_{\bar{z_0},e}(\lambda)-
    \dualunit\varphi_{\bar{z_0},e}(\lambda_0)\|_E
\end{align*}
for $z\in{\disc(z_0,\frac{r}{2})}$ and $\lambda \in \overline{\disc(\lambda_0,\rho)}$, which yields  $\vartheta$ is continuous. By \cite[Lemma 8.9]{Muj} a function is holomorphic if and only if is separately 
holomorphic and continuous, which implies that $\vartheta$ is strongly holomorphic.

Since $\vartheta$ is strongly holomorphic, we infer from \cite[Corollary 15.3.3]{garr} that a
function
\begin{equation*}
    \varOmega\ni z \to (\varOmega^\prime \ni \lambda \to \vartheta(\lambda,z)\in E)\in \bspace
\end{equation*}
is a $\bspace$-valued holomorphic function. Therefore  $ \varOmega\ni z \to \fjad^\prime ({z})e\in \bspace$ for all $e\in E$ is also strongly holomorphic.
By criterion for the holomorphy of operator-valued
functions (see \cite[Theorem 1.7.1]{goha}), $\fjad^\prime$ is an operator-valued holomorphic function.
\end{proof}

Our next goal is to show that
the quintuple ( $\bspace^\prime$,    $\fjad^\prime$, $\lein$, $\mul$) satisfies the conditions (A1)-(A3). The next theorem is inspired by \cite[Proposition 5.2]{ale} (cf. Example \ref{aleman}).


\begin{theorem}\label{dual}
Suppose that the quintuple ($\bspace$,    $\fjad$, $\lein$, $\mul$) is as above. Then the quintuple ( $\bspace^{\prime}$,    $\fjad^\prime$, $\lein$, $\mul$) satisfies the conditions (A1)-(A3). Moreover, 
\begin{itemize}
    \item[(i)]  the following intertwining relations hold
\begin{equation*}
    \lein \dualunit = \dualunit\mul^*\quad\text{and}\quad
    \mul\dualunit = \dualunit \lein^*,
\end{equation*}

\item[(ii)] $\fjad$ and  $\fjad^\prime$ are connected through the relation
    \begin{equation*}
    \la(\fjad^\prime( \bar{z}) e_1) (\lambda),e_2\ra=
  \la e_1,(\fjad( \bar{ \lambda}) e_2)(z)\ra
\end{equation*}
for every $e_1,e_2\in E$, $z\in \varOmega$, $\lambda\in \varOmega^\prime$
\end{itemize}

\end{theorem}
\begin{proof}
Condition (ii) follows from Theorem \ref{holotw}.

Let $K$  be a compact subset of $\varOmega^\prime$. Since  $\fjad$ is continuous, there exist a constant $C_K>0$ such that $\|\fjad(\lambda)\|_{\sbou(E,\bspace)}<C_K$ for $\lambda \in K$. By \eqref{cdu}, we have
\begin{equation*}
     |\la e, (\dualunit\varphi)(\lambda) \ra |=|\varphi(\fjad (\bar{\lambda})e)|\leq \|\varphi\|_{\bspace^*}\|\fjad (\bar{\lambda})e)\|_\bspace
     \leq C_K\|\dualunit\varphi\|_{\bspace^\prime}\|e\|_E,
\end{equation*}
for $e \in E$, $\varphi\in \bspace^*$, $\lambda \in K$, which yields
\begin{equation*}
     \|(\dualunit\varphi)(\lambda)\|
     \leq C_K\|\dualunit\varphi\|,\qquad \varphi\in \bspace^*, \lambda \in K.
\end{equation*}
This proves condition (A1).


We show that the subspace $\lin\{\fjad^\prime (\bar{z})e \colon e\in E, \: z \in \varOmega \}$ is dense in $\bspace^\prime$.
Let $V:=\lin\{\varphi_{z,e}\colon e\in E, \: z \in \varOmega \}$.
Suppose that the subspace  $V$ is not dense in $\bspace^*$. Then there exist $F\in \bspace^{**}$ such that
$F\neq0$ and $F|_V=0$. By reflexivity there exist $f\in \bspace$ such that
\begin{equation*}
    F(\varphi)=\varphi(f), \qquad \varphi \in \bspace^*.
\end{equation*}
This implies that
\begin{equation*}
  F(\varphi_{z,e})=\la f(z),e\ra, \qquad e\in E, z\in \varOmega.
\end{equation*}
Since $F|_V=0$, we deduce that $\la f(z),e\ra = 0$ for every $e\in E$ and $z\in \varOmega$. By Identity theorem $f=0$. This shows that $F=0$ and thus $V$ is dense in $\bspace^*$. An application of \eqref{fjadprdef} completes the proof of property (A2).


We show that $\lein\fjad^\prime (\bar{z}) =  \bar{z}\fjad^\prime (\bar{z})$, for every  $z\in \varOmega$. Since $\lein=\dualunit \mul^*\dualunit^*$, by \eqref{fjadrel} the following equalities hold
\begin{align*}
    \la e_1 ,( \lein\fjad^\prime (\bar{z})e_2)(\lambda)\ra &\overset{\eqref{fjadprdef}}{=} \la e_1, \dualunit\mul^* \varphi_{z,e_2}(\lambda)\ra\overset{\eqref{cdu}}{=}(\mul^* \varphi_{z,e_2})(\fjad (\bar{\lambda})e_1)
    \\
    &\hspace{1ex}= \varphi_{z,e_2}(\mul\fjad (\bar{\lambda})e_1)\overset{\eqref{funct}}{=}\la z(\fjad (\bar{\lambda})e_1) (z), e_2\ra\\
    &\hspace{1ex}= \la e_1, \bar{z}(\fjad^\prime (\bar{z})e_2)(\lambda)\ra
\end{align*}
for every $e_1, e_2\in E$, $z\in \varOmega$ and $\lambda\in \varOmega^\prime$. This shows property (A3).


It remains to prove that $ \mul\dualunit = \dualunit \lein^*$. 
 Combining \eqref{cdu} with (A3), we get
\begin{align*}
   \la e , \dualunit \lein^*\varphi(\lambda)\ra&\overset{\eqref{cdu}}{=}  \lein^*\varphi(\fjad (\bar{\lambda})e)=\varphi(\lein\fjad (\bar{\lambda})e)\overset{(A3)}{=}\varphi(\bar{\lambda}\fjad  (\bar{\lambda})e)\\&\hspace{1ex}= \bar{\lambda}\varphi(\fjad (\bar{\lambda})e) \overset{\eqref{cdu}}{=} \la e, \lambda \dualunit\varphi(\lambda)\ra=\la e, \mul\dualunit\varphi(\lambda)\ra
\end{align*}
for $e \in E$, $\varphi\in\bspace^*$.


\end{proof}



\section{Examples}
In  this  section,  we collect a variety of examples of Banach spaces  that satisfy the conditions (A1)-(A3).
These examples include the classical Banach spaces of holomorphic functions in the unit disc: the Hardy space, the  Bergman space and the
Dirichlet space as well as the Hilbert spaces  of vector-valued analytic functions on an annulus associated with analytic models for  left-invertible operators.


In \cite[Sec. 5]{ale} A. Aleman, S. Richter and W. T. Ross studied the Banach space $\bspace$ of analytic functions on $\disc$ which satisfies the following properties:
\begin{itemize}
    \item[(B1)]$\mul \bspace \subset \bspace$,
    \item[(B2)]  
    $\bspace \hookrightarrow \hol(\disc)$ the inclusion map is both injective and
continuous
(the space $\hol(\disc)$ with the topology of uniform convergence on compact sets),
    
    \item[(B3)] $1\in \bspace$,
    
     \item[(B4)] $\lein_\lambda \bspace \subset \bspace$, 
    \item[(B5)] $\sigma(\mul) = \overline{\disc}$,
    \item[(B6)] the polynomials are dense in $\bspace$,
    \item[(B7)] $\bspace$ is reflexive,
\end{itemize}
where for $\lambda\in \disc$ an operator $\lein_\lambda\colon \bspace \to \bspace$ is given by
\begin{equation*}
    (\lein_\lambda f)(z) = \frac{f(z)-f(\lambda)}{z-\lambda}, \qquad z \in \disc.
\end{equation*}
This example include the classical Banach spaces of holomorphic functions in the unit disc: the Hardy spaces, Bergman spaces and the 
Dirichlet spaces (see \cite[Examples 1.3 and 1.4]{ale}).


\begin{ex}\label{aleman} Let $\bspace$ be a Banach space  of analytic functions on $\disc$ which satisfies the conditions (B1)-(B7). The multiplication operator $\mul$ is left-invertible
and operator 
 $\lein:\bspace \to \bspace$  given by
  \begin{equation*}
      (\lein f)(z) = \frac{f(z)-f(0)}{z},\qquad f \in \bspace
  \end{equation*}
is its left inverse. Let  $\fjad:\disc\rightarrow \sbou(\comp,\hol(\comp))$ be a  linear map defined by
\begin{equation*}
    \fjad(\lambda)\omega:=\omega k_\lambda, \qquad \omega \in \comp,  \lambda \in \disc,
\end{equation*}
where $k_\lambda\colon\comp\to\comp$, $\lambda\in \disc$ is a holomorphic function defined by
\begin{equation*}
    k_\lambda(z)=\frac{1}{1-{\lambda}z},\quad z\in \disc.
\end{equation*}
It is trivial that the quintuple ($\bspace$,    $\fjad$, $\lein$, $\mul$) satisfies conditions (A1) and (A3).
The fact that quintuple satisfies condition (A2) follows from  \cite[Proposition 2.2]{ale}.
\end{ex}

Following \cite[Definition 2.4]{shi}, we say that $T \in \bou$ possesses \textit{the wandering subspace property}, if
\begin{equation*}
    [\nul T^*)]_T = \bigvee\{T^n\nul T^*)\colon n\in\natu \}=\hil.
\end{equation*}
It turns out that for
a left-invertible operator $T$, $T$ is analytic if and only if the Cauchy dual $\cd$ of $T$ possesses wandering subspace property (see \cite[ Proposition 2.7]{shi}). The next two examples are related to the Shimorin's analytic model  and  the model constructed in \cite[Section 3]{ja3}. For the sake of completeness, we only provide definitions of the quintuple ($\chil$, $\fjad$, $\lein$, $\mul$) here, the justification is given in the next section (see Theorem \ref{just}).

\begin{ex}\label{shmod} Let $T\in\bou$ be a left-invertible analytic operator  with the wandering subspace property and $E:=\nul T^*)$.  Let $\chil$ be a Hilbert space of vector-valued analytic functions 
 associated with $T$.  The multiplication operator $\mul$ is left-invertible and  $\lein:\chil \to \chil$  given by
  \begin{equation*}
      (\lein f)(z) = \frac{f(z)-f(0)}{z}, \quad f\in\chil,
  \end{equation*}
is its left inverse. Let $\fjad:\disc\rightarrow \sbou(E,\bspace)$ be an  operator-valued holomorphic function defined by
\begin{equation*}
    \fjad (\lambda):=\sum_{n=0}^\infty{\lambda}^n \mul^n.
\end{equation*}
It turns out that the quintuple ($\chil$,    $\fjad$, $\lein$, $\mul$) satisfies condions (A1)-(A3).
\end{ex}

\begin{ex}\label{ppmod} Let $T\in\bou$ be a left-invertible operator and  $E$ be a closed subspace of $\hil$ such that condition \eqref{dualconv} below holds. Let $\chil$ be a Hilbert space of vector-valued analytic functions 
 associated with $T$,
$\lein:\chil \to \chil$ be a left inverse of $\mul$ given by
\begin{equation*}
    (\mathscr{L}f)(z)=\frac{f(z)-(P_{\nul \mul^*)}f)(z)}{z}, \quad f\in\chil.
\end{equation*}
and $\fjad:(\varOmega^\prime)^*\rightarrow \sbou(E,\bspace)$ be an  operator-valued holomorphic function defined by
\begin{equation*}
 \fjad (\lambda) :=
 \sum_{n=1}^\infty \frac{1}{{\lambda}^n}\mathscr{L}^n +  \sum_{n=0}^\infty{\lambda}^n \mul^n,
\end{equation*}
where $\Omega^\prime$ is as in \eqref{dualconv}.
 In the next section we show that the quintuple ($\chil$,    $\fjad$, $\lein$, $\mul$) satisfies conditions (A1)-(A3).

\end{ex}
\section{Duality for analytic model}

 In this section, we show that the analytic  model for a left-invertible operator $T$ is a natural example of a Banach space of vector-valued analytic functions considered in Section \ref{dualitysec}. We will describe the relationship between the analytic model for $T$ and the analytic model for the Cauchy dual operator $T^\prime$. 
 
The Cauchy dual operator $\cd$ of a left-invertible operator  is itself left-invertible. Assume now that there exist a closed subspace $E\subset\hil$ such that $\gwon=\hil$ and $[E]_{\cd,T}=\hil$ hold. Then for both operators $T$ and $\cd$ one can construct Hilbert spaces $\chil$ and $\chil^\prime$  of $E$-valued Laurent series. Therefore, by \eqref{mod} $\chil'$ is the space of Laurent series of the form $U_x'$, $x\in\hil$, where 
\begin{equation}\label{cddualser}
    U^\prime_x(z) :=\sum_{n=1}^\infty
(P_ET^{\prime n}x)\frac{1}{z^n}+
\sum_{n=0}^\infty
(P_ET^{*n}x)z^n.
\end{equation}

To avoid repetition, we state the following assumption which will be used
frequently in this section.
   \begin{align} 
   \label{dualconv}
\begin{minipage}{70ex} 
The operator  $T\in \bou$ is left-invertible   and $E$ is a closed subspace of $\hil$ such that $\gwon=\hil$ and $[E]_{\scd,T}=\hil$. Suppouse that the series \eqref{mod} and \eqref{cddualser} are convergent in $E$ on an annulus $\varOmega:=\ann(r^-,r^+)$ and $\varOmega^\prime:=\ann(r^{\prime-},r^{\prime+})$ respectively, where  $0\leq r^-<r^+$ and $0\leq r^{\prime-}<r^{\prime+}$.
\end{minipage}
    \end{align}
We will be consider the quintuple ($\chil$, $\fjad$, $\lein$, $\mul$), where
\begin{itemize}
    \item $\chil$ is a Hilbert space of vector-valued analytic functions 
 associated with $T$,
 
  \item $\mul:\chil \to \chil$ is a multipliction operator,
 
 \item $\lein:\chil \to \chil$ is a left inverse of $\mul$ given by
\begin{equation*}\label{dualr}
    (\mathscr{L}f)(z)=\frac{f(z)-(P_{\nul \mul^*)}f)(z)}{z}, \quad f\in\chil,
\end{equation*}

    \item $\fjad:(\varOmega^\prime)^*\rightarrow \sbou(E,\bspace)$ is an operator-valued holomorphic function defined by
    \begin{equation}\label{fjadt}
 \fjad (\lambda)  := \sum_{n=1}^\infty \frac{1}{{\lambda}^n}\mathscr{L}^n +  \sum_{n=0}^\infty{\lambda}^n \mul^n.
\end{equation}

\end{itemize}
The following lemma shows that $\fjad$ is well-defined.
\begin{lemma}\label{welldef} Suppose  \eqref{dualconv} holds. Then the following conditions hold:
\begin{itemize}
\item[(i)] The series in \eqref{fjadt} converges absolutely and uniformly in operator norm on any compact subset contained in $\ann(r^{\prime-},r^{\prime+})$.
    \item[(ii)] The function $\fjad$ is well-defined and holomorphic on $\ann(r^{\prime-},r^{\prime+})$.
\end{itemize}
\end{lemma}
\begin{proof}
(i) By \cite[Theorem 3.8]{ja3} with $T^\prime$ in place of $T$  the series
\begin{equation}\label{series38}
\sum_{n=1}^\infty
(P_ET^{\prime n})\frac{1}{\lambda^n}+   \sum_{n=0}^\infty
(\pe{T^{*n}})\lambda^n \in \boldsymbol B(\hil,E)
\end{equation}
converges absolutely and uniformly in operator norm on any compact subset contained in $\ann(r^{\prime-},r^{\prime+})$. Therefore, the series 
\begin{equation*}
\sum_{n=1}^\infty
(T^{\prime* n}P_E)\frac{1}{\lambda^n}+   \sum_{n=0}^\infty
({T^{n}}\pe)\lambda^n \in \boldsymbol B(E,\hil)
\end{equation*}
also converges absolutely and uniformly in operator norm on any compact subset contained in $\ann(r^{\prime-},r^{\prime+})$. This combined with the fact that the operators $T$, $\scd$ are unitarily equivalent to the operators $\mul$, $\mathscr{L}$ respectively, completes the proof.

(ii) This is a direct consequence of (i).
\end{proof}

We now show that the quintuple ($\chil$, $\fjad$, $\lein$, $\mul$)
satisfies properties (A1)- (A3).





\begin{theorem}\label{just}
Suppose  \eqref{dualconv} holds. Then
the quintuple ($\chil$, $\fjad$, $\lein$, $\mul$)
satisfies properties (A1)-(A3), that is
\begin{itemize}
    \item[(i)] the inclusion map 
\begin{equation*}
   \iota :\chil  \hookrightarrow \hol( \ann(r^-,r^+), E),
 \end{equation*}
is both injective and
continuous,
where  $\hol( \ann(r^-,r^+), E)$ is with the topology of uniform convergence on compact sets 
\item[(ii)]  the subspace $\lin \{\fjad( \bar{\lambda})e \colon e\in E, \: \lambda \in \varOmega^\prime\}$ is dense in $\bspace$,
\item[(iii)] $\mathscr{L} \fjad ( \bar{\lambda} )= \bar{\lambda}\fjad (\bar{\lambda} )$.
\end{itemize}

\end{theorem}
\begin{proof}
(i) Since by \cite[Theorem 3.8]{ja3} the series
\begin{equation*}
\sum_{n=1}^\infty
(P_ET^{n})\frac{1}{\lambda^n}+   \sum_{n=0}^\infty
(\pe{T^{*\prime n}})\lambda^n \in \boldsymbol B(\hil,E)
\end{equation*}
converges absolutely and uniformly in operator norm on any compact set contained in $\ann(r^-,r^+)$ there exist constant $C_K>0$ for every compact subset $K\subset \ann(r^-,r^+)$ such that
\begin{equation*}
\|\sum_{n=1}^\infty
(P_ET^{n})\frac{1}{\lambda^n}+   \sum_{n=0}^\infty
(\pe{T^{*\prime n}})\lambda^n \|\leq C_K, \quad \lambda \in K.
\end{equation*}
This implies that 
\begin{align*}
     \|Ux(\lambda)\|_E &=\|\sum_{n=1}^\infty
(P_ET^{n}x)\frac{1}{\lambda^n}+
\sum_{n=0}^\infty
(P_E{T^{\prime*n}}x)\lambda^n\|_E\\
&\leq C_K\|x\|_\hil=C_K\|Ux\|_\chil
\end{align*}
for $x\in\hil$ and $\lambda\in K$. Therefore, the inclusion map $\iota$
is 
continuous in the topology of uniform convergence on compact sets.

(ii) Suppouse that there exist  $f \in \chil$ such that 
\begin{equation*}\label{dens}
    \la \fjad (\bar{ \lambda})e, f\ra =0, \qquad  e \in E, \lambda \in \ann(r^{\prime-},r^{\prime+}).
\end{equation*}
This is equivalent to 
\begin{equation*}
    \sum_{n=1}^\infty
(P_ET^{\prime n}U^*f)\frac{1}{\lambda^n}+
\sum_{n=0}^\infty
(P_E{T^{*n}}U^*f)\lambda^n = 0, \quad \lambda\in\ann(r^{\prime-},r^{\prime+}).
\end{equation*} An application of \cite[Lemma 3.1]{ja3} with $T^\prime$ in place of $T$ completes the proof of assertion (ii).

(iii) By \cite[Theorem 3.2]{ja3}, we have 
\begin{align*}
\lein &= U\scd,\\
     \fjad(\lambda) &=  U(\sum_{n=1}^\infty \frac{1}{{\lambda}^n}T^{\prime * n} +  \sum_{n=0}^\infty{\lambda}^n T^n).
\end{align*}
An easy calculation shows that (iii) holds.
\end{proof}





Now, we show that both
Hilbert space constructed for the Cauchy dual operator in \eqref{mod}  and the Cauchy dual space obtained in construction \eqref{cdu} coincides.

\begin{theorem}\label{chp}
Suppouse that \eqref{dualconv} holds.  Then the Hilbert space constructed for the Cauchy dual operator $T^\prime$ in \eqref{mod}  coincide with the Cauchy dual space obtained in construction \eqref{cdu} for the quintuple ($\chil$, $\fjad$, $\lein$, $\mul$). Moreover, if $\varphi \in \chil^*$ is  represented by $g\in \chil$, that is,
\begin{align*}
    \varphi(f)=\la f,g\ra,\quad f\in\chil,
\end{align*}
 then $\dualunit \varphi = U^{\prime}U^*g$.

\end{theorem}
\begin{proof}
Note that
\begin{align*}
  \varphi(\fjad  (\bar{\lambda})e)&=\la \fjad  (\bar{\lambda})e, g\ra_\chil= \la U(\sum_{n=1}^\infty \frac{1}{\bar{\lambda}^n}T^{\prime * n}e +  \sum_{n=0}^\infty\bar{\lambda}^n T^ne), g\ra_\chil\\
  &=\la \sum_{n=1}^\infty \frac{1}{\bar{\lambda}^n}T^{\prime * n}e +  \sum_{n=0}^\infty\bar{\lambda}^n T^ne, U^*g\ra_\hil \notag\\&=\la e, \sum_{n=1}^\infty
(P_ET^{\prime n}U^*g)\frac{1}{\lambda^n}+
\sum_{n=0}^\infty
(P_E{T^{*n}}U^*g)\lambda^n\ra_\hil\notag.
\end{align*}
for every $e \in E$, $ \lambda \in \ann(r^{\prime-},r^{\prime+})$.
Therefore
\begin{align*}
    \dualunit \varphi (\lambda) &= \sum_{n=1}^\infty
(P_ET^{\prime n}U^*g)\frac{1}{\lambda^n}+
\sum_{n=0}^\infty
(P_E{T^{*n}}U^*g)\lambda^n
\\&=(U^\prime U^*g)(\lambda),
\end{align*}
for $\lambda \in 
    \ann(r^{\prime-},r^{\prime+})$. This completes the proof.
\end{proof}



Our next aim is to characterise  when the duality between $\chil$ and $\chil^\prime$ obtained by identifying
them with $\hil$ is the same as the duality obtained from the Cauchy pairing. Let us
point out that the Cauchy pairing in  \eqref{shimcaupar} 
 is  between two $E$-valued polynomials.
Note that if left invertible operator possesses the  wandering subspace property, then   $E$-valued polynomials are dense in
$\chil$. 
Therefore, in order to obtain an analogue of \eqref{shimcaupar} we replace $E$-valued polynomials with a dense subspace  $\mathcal{V}:=\chil\cap\hol(\ann(r^-,\infty),E)$ of $\chil$, which includes polynomials. First, we prove the following auxiliary lemma.

\begin{lemma}\label{granlemma}
Suppose that  \eqref{dualconv} holds, $r^-<1$ and $r^{\prime-}\Le 1\Le r^{\prime+}$. Then
there exist an open neighbourhood $W\subset(0,\infty)$ of $1$ such that the series below
is convergent absolutely for every $r\in W$, $f\in \mathcal{V}:=\chil\cap\hol(\ann(r^-,\infty),E)$ and
\begin{align}\label{limser}
   \lim_{r\to1}       \int_0^{2\pi} (\fjad(re^{it})) (f(re^{it}))\frac{dt}{2\pi} &  =
     \lim_{r\to1}  U\big(\sum_{n=1}^\infty \frac{1}{r^{2n}}T^{\prime *n}\pe T^nU^*f \\&+\notag \sum_{n=0}^\infty r^{2n}T^n\pe T^{\prime *n}U^*f \big)\\&=U\big(\sum_{n=1}^\infty T^{\prime *n}\pe T^nU^*f + \sum_{n=0}^\infty T^n\pe T^{\prime *n}U^*f\big).\notag
\end{align}
\end{lemma}
\begin{proof}
Take $\rho_1\in(r^-,\infty)$ and $\rho_2\in(r^{\prime-},r^{\prime+})$ such that $\rho_1\rho_2>1$. 
By \cite[Theorem 3.8]{ja3} the  series  \eqref{series38} is absolutely convergent in $E$ on an annulus $\ann(r^{\prime-},r^{\prime+})$ thus there exists a constant $M_1>0$ such that
 \begin{equation*}
   \| \pe T^{*n}\rho_2^n\|\Le M_1, \qquad  n\in \natu.
\end{equation*}
Since the series \eqref{mod} with $U^*f$ in place of $x$ is convergent there exist a  constant $D_1>0$, such that
\begin{equation*}
    \| \pe T^{\prime*n}U^*f\rho_1^n\|\Le D_1, \qquad n\in \natu.
\end{equation*}
This implies that
\begin{align*}
    \|T^n\pe T^{\prime*n}U^*f\|\Le   \| T^n\pe\|\| \pe T^{\prime*n}U^*f\|\Le \frac{M_1D_1}{(\rho_1\rho_2)^n}.
\end{align*}
Therefore, the series
$\sum_{n=0}^\infty T^n\pe T^{\prime*n}U^*f z^n$ is absolutely  convergent on $\disc(\rho_1\rho_2)$ and we see that
\begin{align}\label{lim1}
     \lim_{r\to1}   \sum_{n=0}^\infty r^{2n}T^n\pe T^{\prime *n}U^*f&= \sum_{n=0}^\infty T^n\pe T^{\prime *n}U^*f.
\end{align}
Take $\rho_3\in(r^-,\infty)$ and $\rho_4\in(r^{\prime-},r^{\prime+})$ such that $\rho_3\rho_4<1$. ).
Similarly, we see  that there exist  constants $M_2,D_2>0$,  such that
\begin{equation*}
    \| \pe T^nU^*f\frac{1}{\rho_3^n}\|\Le D_2,\quad \| \pe T^{\prime n}\frac{1}{\rho_4^n}\|\Le M_2, \qquad n\in \natu.
\end{equation*}
As a consequence, we have
\begin{align*}
    \|T^{\prime *n}\pe T^nU^*f\| \Le  \|T^{\prime *n}\pe\|\|\pe T^nU^*f\| \Le {C_2D_2}{(\rho_3\rho_4)^n}.
\end{align*}
We see that the series
$\sum_{n=1}^\infty T^{\prime*n}\pe T^nU^*f\frac{1}{z^n}$ converges in $\ann( {\rho_3\rho_4},\infty)$ and
\begin{align*}
     \lim_{r\to1}  \sum_{n=1}^\infty \frac{1}{r^{2n}}T^{\prime*n}\pe T^nU^*f =\sum_{n=1}^\infty T^{\prime*n}\pe T^nU^*f.
\end{align*}
This combined with
 \eqref{lim1} gives the second equality in \eqref{limser}.
 
 Since  $T^{\prime *}$ is unitarily equivalent to  $\mathscr{L}$ (see Section \ref{anamodel}), we have
\begin{align*} 
    \frac{1}{\bar{\lambda}^m}\mathscr{L}^m (f(\lambda))  = \frac{1}{\bar{\lambda}^m} UT^{\prime *m}(\sum_{n=1}^\infty
(P_ET^{n}U^*f)\frac{1}{\lambda^n}+
\sum_{n=0}^\infty
(P_E{T^{\prime*n}}U^*f)\lambda^n)
\end{align*}
for $\lambda \in \ann(r^-,\infty)$, $m\in\natu$. As a consequence, we get
\begin{equation}\label{rum}
\frac{1}{2\pi}   \int_{0}^{2\pi}  \frac{1}{r\overline{e^{itm}}}\mathscr{L}^m (f(re^{it}))dt =\frac{1}{r^{2m}} UT^{\prime *m} P_ET^{m}U^*f, 
\end{equation}
for $r\in(r^-,\infty)$ and $m\in\natu$. Similarly, we obtain that 
\begin{align*}
   {\bar{\lambda}}^m z^m f(\lambda)  =   {\bar{\lambda}^m} UT^m(\sum_{n=1}^\infty
(P_ET^{n}U^*f)\frac{1}{\lambda^n}+
\sum_{n=0}^\infty
(P_E{T^{\prime*n}}U^*f)\lambda^n)
\end{align*}
for $\lambda\in \ann(r^-,\infty)$, $m\in\natu$. Hence, we have
\begin{equation*}
   \int_0^{2\pi}    {r^m\overline{e^{itm}}}z^m f(re^{it})\frac{dt}{2\pi}   = r^{2m}UT^m
P_ET^{\prime*m}U^*f
\end{equation*}
for $r\in(r^-,\infty)$ and  $m\in\natu$.
This, combined with \eqref{rum}, Lemma \ref{welldef} and changing order of summation and integration yields 
\begin{align*}
     \int_0^{2\pi} (\fjad(re^{it})) (f(re^{it}))\frac{dt}{2\pi} &  =     \int_0^{2\pi} ( \sum_{n=1}^\infty \frac{1}{r^n\overline{{e^{itn}}}}\mathscr{L}^n +  \sum_{n=0}^\infty r^n\bar{e^{itn}} z^n)f(re^{it})\frac{dt}{2\pi} \\
    &=   U(\sum_{n=1}^\infty \frac{1}{r^{2n}}T^{\prime *n}\pe T^nU^*f + \sum_{n=0}^\infty r^{2n}T^n\pe T^{\prime *n}U^*f).
\end{align*}
This  gives the first equality in \eqref{limser} and completes the proof.

\end{proof}

We are now in a position to  prove the main
theorem of this section.

 \begin{theorem}\label{charak}Suppose that  \eqref{dualconv} holds, $r^-<1$ and $r^{\prime-}\Le 1\Le r^{\prime+}$.
Then the subspace $\mathcal{V}:=\chil\cap\hol(\ann(r^-,\infty),E)$ is dense in $\chil$, the limits in (i) and (ii) exist, the series in (iii) and (iv) converges and
 the following conditions are equivalent:
\begin{itemize}
    \item[(i)]
    \begin{equation*}
        f=\lim_{r\to1}
    \int_0^{2\pi} (\fjad(re^{it})) (f(re^{it})) \frac{dt}{2\pi}, \qquad f\in  \mathcal{V},
    \end{equation*}
    \item[(ii)]  \begin{equation*}
    \varphi(f)=\lim_{r\to1}\int_0^{2\pi}\la f(re^{it}),\dualunit\varphi(re^{it})\ra\frac{dt}{2\pi}, \qquad  \varphi \in \chil^*,f\in  \mathcal{V},
\end{equation*}
 \item[(iii)] 
    \begin{equation*}
        \la U^*f, U^{\prime*}g\ra_\hil = \sum _{n=-\infty}^\infty \la \hat{f}(n), \hat{g}(n)\ra_E,\qquad  g\in 
    \chil^\prime,f\in  \mathcal{V},
    \end{equation*}

    \item[(iv)]
    \begin{equation*}
        \sum_{n=1}^\infty T^{\prime *n}\pe T^nU^*f + \sum_{n=0}^\infty T^n\pe T^{\prime *n}U^*f  = U^*f,\quad f\in  \mathcal{V}.
    \end{equation*}

\end{itemize}

\end{theorem}
\begin{proof}
First, we prove that the subspace $\mathcal{V}$ is dense in $\chil$. Note that the series \eqref{mod} is convergent in $E$ on an annulus $\ann(r^-,r^+)$. Hence, we see that the series
 \begin{align}\label{osobl}
    UT^{\prime*m}e&=\sum_{n=1}^\infty
(P_ET^{n}T^{\prime*m}e)\frac{1}{z^n},\qquad m\in \natu, e \in E
\end{align}
 is convergent on an annulus $\ann(r^-,\infty)$. Thus $UT^{\prime*m}e\in\mathcal{V}$. Observe that
 \begin{align*}
    UT^me&=z^me,\qquad m\in \natu
\end{align*}
and $UT^me\in\mathcal{V}$.
 Since $[E]_{\scd,T}=\hil$, this and \eqref{osobl}  shows
 that $\mathcal{V}$ is dense in $\chil$.
 
 It follows from Lemma \ref{granlemma} that the limit in (i) exist and the series in  (iv) converges.
 Fix any $\varphi \in \chil^*$ and  $f\in\chil$.  Note that by \eqref{cdu}, we have
\begin{align}\label{onfunctional}
 \varphi(
    \int_0^{2\pi} (\fjad(re^{it})) (f(re^{it}))\frac{dt}{2\pi})
    &=
    \int_0^{2\pi} \varphi((\fjad(re^{it})) (f(re^{it})))\frac{dt}{2\pi}
    \\&=  \int_0^{2\pi}\la f(re^{it}),\dualunit\varphi(re^{it})\ra\frac{dt}{2\pi}.\notag
\end{align}
This combined with Lemma \ref{granlemma} shows that the limit in (ii) exist.

It follows from \eqref{gg} that 
 \begin{equation}\label{eqsze}
        \sum _{n=-\infty}^\infty r^{2n}\la \hat{f}(n), \hat{g}(n)\ra
        =\sum_{n=1}^\infty \frac{1}{r^{2n}}\la T^{\prime *n}\pe T^nU^*f,g\ra + \sum_{n=0}^\infty r^{2n}\la T^n\pe T^{\prime *n}U^*f ,g\ra
    \end{equation}
for $f\in  \mathcal{V}$, $g\in \chil^\prime$ and $r\in W$, where $W$ is as in Lemma \ref{granlemma}. By Lemma \ref{granlemma}, we see that that  the series in  (iii) converges.

(i) $\implies$ (ii) It follows from \eqref{onfunctional}.

(ii)$\implies$ (iii)
Combining  \eqref{eqsze} with Lemma \ref{granlemma}, we deduce that
the following limit exists and
\begin{equation}\label{rlimfg}
 \lim_{r\to1}\sum_{n=-\infty}^\infty r^{2n}\la 
\hat{f}(n) ,\hat{g}(n)\ra= \sum_{n=-\infty}^\infty \la 
\hat{f}(n) ,\hat{g}(n)\ra,\quad f\in  \mathcal{V},\: g\in \chil^\prime.
\end{equation}
Fix $g\in \chil^\prime$. Let $\varphi \in \chil^*$ be defined by
\begin{equation*}
    \varphi(f)=\la f, UU^{\prime*}g\ra,\quad f \in\chil.
\end{equation*}
Then (ii) combined with \eqref{rlimfg}, Lemma \ref{granlemma} and Theorem \ref{chp} implies that
\begin{align}\label{swl}
\la U^*f, U^{\prime*}g\ra&=\la f, UU^{\prime*}g\ra =\varphi(f)=\lim_{r\to1}\int_0^{2\pi}\la f(re^{it}),\dualunit\varphi(re^{it})\ra\frac{dt}{2\pi}\\
&=\notag
\lim_{r\to1}\int_0^{2\pi}\la f(re^{it}),g(re^{it})\ra\frac{dt}{2\pi}
\\&=\notag
\lim_{r\to1}\int_0^{2\pi}\la \sum_{n=-\infty}^\infty
\hat{f}(n)(re^{it})^n ,\sum_{n=-\infty}^\infty\hat{g}(n)(re^{it})^n\ra\frac{dt}{2\pi}
\\&=\notag
\lim_{r\to1}\sum_{n=-\infty}^\infty\sum_{m=-\infty}^\infty\int_0^{2\pi}\la 
\hat{f}(n)(re^{it})^n ,\hat{g}(m)(re^{it})^m\ra\frac{dt}{2\pi}
\\&=\notag
\lim_{r\to1}\sum_{n=-\infty}^\infty r^{2n}\la 
\hat{f}(n) ,\hat{g}(n)\ra= \sum_{n=-\infty}^\infty \la 
\hat{f}(n) ,\hat{g}(n)\ra.
\end{align}

(iii)$\Leftrightarrow$(iv)
Combining \eqref{eqsze} with the fact that both series in (iii) and (iv) are convergent completes the proof of equivalence (iii)$\Leftrightarrow$(iv).

(iv)$\implies$ (i) Using Lemma \ref{granlemma}, we obtain
\begin{align*}
   \lim_{r\to1}       \int_0^{2\pi} (\fjad(re^{it})) (f(re^{it}))\frac{dt}{2\pi} =U\big(\sum_{n=1}^\infty T^{\prime *n}\pe T^nU^*f + \sum_{n=0}^\infty T^n\pe T^{\prime *n}U^*f\big)=f
\end{align*}
for $f \in \mathcal{V}$,
which completes the proof.

\end{proof}



\section{Weighted composition operators}
In this section, we illustrate Theorem \ref{charak}
by considering
examples of composition operators.  Since the analytic structure of composition operators plays a major role in this section,
we outline it in the following discussion.  Let $X$ be a countable set, $w:X\to \comp$ be a complex function,  $\varphi:X\to X$ be a transformation of $X$, which has finite branching index and $\com\in \sbou(\ell^2(X))$ be a weighted composition operator.
We  only consider
composition functions with one orbit, since an orbit induces a reducing subspace
to which the restriction of the weighted composition operator is again a weighted
composition operator.
Note that any self-map $\varphi:X\to X$ induces  a directed graph $(X, E^\varphi)$  given by 
 \begin{equation}\label{graph}
     E^\varphi:=\set{(x,y)\in X\times X \colon x=\varphi(y)}.
 \end{equation}
  Perhaps it is appropriate at this point to note that a self-map with one orbit can have at most one cycle.
The directed graph $(X, E^\varphi)$ is a directed three in the case of when $\varphi$ has one orbit and does not have a cycle. 
The next lemma shows that in the case of  rootless
directed tree with finite branching index
  there exist some special vertex.
\begin{lemma}[{\cite[Lemma 6.1]{chav}}]\label{groot}
 Let $\ttt=(V,E)$  be a rootless directed tree with finite branching
index $m$. Then there exist a vertex $\varOmega\in V_\prec$ such that
\begin{equation}\label{root}
\card(\czil(\parr^{(n)}(\varOmega)))=1,\qquad n\in\mathbb{Z}_+.
\end{equation}
Moreover, if $V_\prec$ is non-empty, then there exists a unique $\varOmega\in V_\prec$ satisfying \eqref{root}.
 \end{lemma} The vertex $\varOmega\in V_\prec$ appearing in the statement of Lemma \ref{groot} is called
 \textit{generalized root}. We put $x^*:=\parr(\varOmega)$ in the definition of function $[\varphi]:X\rightarrow \mathbb{Z}$ for orbit $F$ of $\varphi$
not containing a cycle (see Section \ref{prel}).

The following lemma describes a  subspace $E\subset \ell^2(X)$, which satisfies condition \eqref{prep} with $\com$ in place of $T$.
\begin{lemma}[{\cite[Lemma 4.2]{ja3}}]\label{wond} Let $X$ be a countable set, $w:X\to \comp$ be a complex
function on $X$ and  $\varphi:X\to X$ be a transformation of $X$,  which  has  finite  branching  index. Let $\com$ be a weighted composition operator in $\ell^2(X)$ and
\begin{equation}\label{eee}
E: = \left\{ \begin{array}{ll}
\bigoplus_{x\in\gen_\varphi(1,1)}\la e_x\ra\oplus\nul (\com|_{\ell^2(\des(x))})^*) & \textrm{when $\varphi$ has  a  cycle,}
\\
\la e_\varOmega\ra\oplus\nul\com^*) & \textrm{otherwise,}
\end{array} \right.
\end{equation}
where  $\des(x):=\bigcup_{n=0}^\infty\varphi^{(-n)}(x)$
and $\varOmega$ is a generalized root of the tree. Then the subspace
$E$ has
the following properties:
\begin{itemize}
\item[(i)]$[E]_{{\cdcom},{\com^*}}=\hil$ and $[E]_{{\com},{\cdcom^*}}=\hil$,
\item[(ii)] $E\perp \com^n E$ and $E\perp \cdcom^n E$,  $n\in \mathbb{Z}_+$.
\end{itemize}
\end{lemma}

 Suppose that the series \eqref{mod} with $\com$ in place of $T$ is convergent in $E$ on an annulus $\ann(r^-,r^+)$ with $r^-<r^+$ and  $r^-,r^+\in[0,\infty)$ for every $x\in\hil$. In \cite[see (4.7) and (4.8)]{ja3} the inner and outer radius of convergence for weighted composition operator was described only  in terms of its
weights. In this case (see \cite[Theorem 4.3]{ja3}), there exist a $z$-invariant reproducing kernel Hilbert
space $\mathscr{H}$ of $E$-valued holomorphic functions defined on the annulus $\ann(r^-,r^+)$ and a unitary
mapping $U:\ell^2(V)\rightarrow\mathscr{H} $ such that $\mathscr{M}_zU=U\com$, where $\mathscr{M}_z$ denotes the operator
of multiplication by $z$ on $\mathscr{H}$. Moreover, in the case when $\varphi$ does not have a cycle the linear subspace generated by
$E$-valued polynomials in
$z$
and $\tilde{E}$-valued polynomials involving only 
negative powers of $z$ is dense in $\chil$, that is
\begin{equation*}\label{gestosc}
\bigvee(\{z^nE\colon n\in \natu\}\cup\{\frac{1}{z^n}\tilde{E} \colon n\in \mathbb{Z}_+\})=\chil,
\end{equation*}
where $\tilde{E}:=\bigvee\{e_x\colon x\in \gen_\varphi(1,1)\}$. If $\varphi$
 has a cycle $\cycle$, then there exist $\tau$ functions $f_1,\dots, f_\tau$ on $\ann(r^-,r^+)$ given by the following Laurent series
 \begin{equation*}
   f_i(z):=\sum_{k=0}^\infty\sum_{j=1}^{\tau} \Lambda^kA_j^i\frac{1}{z^{k\tau+j}}, \quad i=1,...,\tau,
 \end{equation*}
where $\tau:=\card\cycle$, $\Lambda:=\prod_{x\in\cycle}w(x)$ and  $A_j^i\in \tilde{E}$, $ i,j=1,...,\tau$ such that the linear subspace generated by $E$-valued polynomials in $z$ and the above functions is dense in $\chil$, that is
    \begin{equation*}
\bigvee(\{z^nE\colon n\in \natu\}\cup\{f_i \colon i\in\set{1,\dots \tau}\})=\chil.
\end{equation*}
Recently, the analytic structure of   weighted composition operators  and related operators, like weighted shifts on directed trees was studied by several authors (see \cite{ planeta,bdp, chav, ja3}).

We begin by proving that in the case of left-invertible weighted composition operators on $\ell^2(X)$ the duality between $\chil$ and $\chil^\prime$ obtained by identifying
them with $\ell^2(X)$ is the same as the duality obtained from the Cauchy pairing for dense subspace $ \chil\cap\hol(\ann(r^-,\infty),E)$, which contain all vector-valued polynomials.
\begin{theorem}\label{onden}
Let $X$ be a countable set, $w:X\to \comp$ be a complex
function on $X$ and  $\varphi:X\to X$ be a transformation of $X$,  which  has  finite  branching  index. Let $\com$ be a  weighted composition operator in $\ell^2(X)$ such that \eqref{dualconv} holds with $\com$ in place of $T$ and $E$ as in \eqref{eee}. Suppose that $r^-<1$ and $r^{\prime-}\Le 1\Le r^{\prime+}$. Then the duality between $\chil$ and $\chil^\prime$ obtained by identifying
 them with $\ell^2(X)$ is the same as the duality obtained from the Cauchy pairing 
\begin{equation*}\label{par}
    \la U^{-1}f,U^{\prime-1}g\ra_{\ell^2(X)}=\sum_{n=-\infty}^{\infty}\la \hat{f}(n),\hat{g}(n)\ra_E
\end{equation*}
for $f\in \chil\cap\hol(\ann(r^-,\infty),E)$ and $g\in \chil^\prime$.

\end{theorem}
\begin{proof}
Set $T:=\com$. Suppose that $\varphi$ has a cycle. Let $s:\cycle \rightarrow [0,\infty)$ be a function defined by
\begin{equation*}
    s(x):=\sum_{y\in\varphi^{-1}(\varphi(x))}|w(y)|^2.
\end{equation*}
Fix any $x\in \cycle$. Let $h:\{0,1,\dots,\tau \}\to [0,\infty)$ be a function given by
\begin{equation*}
 h(m) := \left\{ \begin{array}{ll}
   \prod_{k=m}^{\tau-1}(w(\varphi^{(k)}(x))w^\prime(\varphi^{(k)}(x)))= \prod_{k=m}^{\tau-1}\frac{|w(\varphi^{(k)}(x))|^2}{s(\varphi^{(k)}(x))}  & \textrm{if $m<\tau$},
\\
   1  & \textrm{if $m=\tau$,}
\end{array} \right.
\end{equation*}
where $w^\prime$ is as in Lemma \ref{cdcom}.
%
Let  $m\in \natu$ be such that $0\leq m<\tau$. We claim that 

\begin{equation}\label{zmudne}
    T^{\prime *n}P_{\lin\{e_y\}} T^n e_x = \left\{ \begin{array}{ll}
h(0)^{l}\frac{|w(y)|^2}{s(\varphi^{(m)}(x))}h(m+1) e_x    & \textrm{if $n=(l+1)\tau-m$, $l\in \natu$},
\\
   0  & \textrm{otherwise,}
\end{array} \right.
\end{equation}
for $y\in\varphi^{-1}(\varphi^{m+1}(x))\setminus\{\varphi^{(m)}(x)\}$.
Indeed, by Lemma \ref{podst} if $v,y\in 
\gen_\varphi{(1,1)}$ and $n\in \cal_+$ then $T^ne_{v} \in \lin\{e_u\colon u \in
\gen_\varphi{(n+1,n+1)}\}$ and, consequently,  $P_{\lin\{e_y\}}T^ne_{v}=0$.
Combining this fact with  Lemma \ref{podst}, we deduce that
\begin{align}\label{ros1}
    P_{\lin\{e_y\}}T^{n}e_{\varphi^k(x)}&= P_{\lin\{e_y\}}T^{n-1}\sum_{u\in\varphi^{-1}(\varphi^k(x))}w(u)e_u 
    \\&= P_{\lin\{e_y\}}T^{n-1}\big(w(\varphi^{k-1}(x))e_{\varphi^{k-1}(x)}\notag
    \\&+ \sum_{u\in\varphi^{-1}(x)\setminus \{\varphi^{\tau-1}(x)\}}w(u)e_u\big) \notag
    \\&= w(\varphi^{k-1}(x))P_{\lin\{e_y\}}T^{n-1}e_{\varphi^{k-1}(x)}\notag
\end{align}
for $y\in\gen_\varphi(1,1) $, $n\Ge2$ and $k\Ge 1$.
Similarly, we obtain
\begin{align}\label{ros2}
    P_{\lin\{e_y\}}T^{n}e_{x}&=  w(\varphi^{\tau-1}(x))P_{\lin\{e_y\}}T^{n-1}e_{\varphi^{\tau-1}(x)}
\end{align}
for $y\in\gen_\varphi(1,1)$ and $n\Ge 2$. 
\begin{align}\label{ros3}
    P_{\lin\{e_y\}}Te_{x}&=  \left\{ \begin{array}{ll}
{w(y)} e_y    & \textrm{if $y\in\varphi^{-1}(x) $},
\\
   0  & \textrm{otherwise,}
\end{array} \right.
\end{align}
for $y\in\gen_\varphi(1,1)$. Combining \eqref{ros1}, \eqref{ros2} and \eqref{ros3}, we deduce that

\begin{equation*}\label{zmudnepom}
   P_{\lin\{e_y\}} T^n e_x = 
 w(y) \prod_{k=0}^{\tau-1}(w(\varphi^{(k)}(x)))^{l}\prod_{k=m+1}^{\tau-1}(w(\varphi^{(k)}(x)) e_y
\end{equation*}
for $y\in\varphi^{-1}(\varphi^{m+1}(x))\setminus\{\varphi^{(m)}(x)\}$, $n=(l+1)\tau-m$, $l\in \natu$ and  $P_{\lin\{e_y\}} T^n e_x =0$ in the other case.  
This and Lemma \ref{cdcom} gives  \eqref{zmudne}.
 Our next goal is to show that the following equality holds
  \begin{equation}\label{odt}
        \sum_{n=1}^\infty T^{\prime *n}\pe T^ne_x + \sum_{n=0}^\infty T^n\pe T^{\prime *n} e_x= e_x,\quad x\in Y, 
        \end{equation}
 where
\begin{equation*} Y:=\left\{\begin{array}{ll}
 \cycle & \textrm{ when $\varphi$ has a cycle}\\
 \set{x\colon [\varphi](x)\Le 0}&\textrm{ otherwise.}
\end{array} \right.
\end{equation*}
 
We now consider two disjunctive cases which cover all possibilities.
First we consider the case when $\varphi$
does not have a
cycle. Fix $[\varphi](x) <0$. Using  Lemmas \ref{cdcom} and \ref{wond}, one can  verify that  $T^n\pe T^{\prime *n}e_x=0$ for every $n \in \natu$ and
\begin{equation}\label{dnhc}
  T^{\prime *n}\pe T^n   e_x = \left\{ \begin{array}{ll}
e_x     & \textrm{if $n=-[\varphi](x)+1$,}
\\
   0  & \textrm{otherwise,}
\end{array} \right.
\end{equation}
which completes the proof
of the case   when $\varphi$
does not have a
cycle.

 It remains to consider the other case when $\varphi$ has a cycle. Fix any $x\in \cycle$. 
%
Define the subspace $E_m := \bigvee \{e_y\colon y\in\varphi^{-1}(\varphi^{m+1}(x))\setminus\{\varphi^{(m)}(x)\} \}$ for every  $0\leq m<\tau$. It follows from \eqref{zmudne} that
\begin{align}\label{geoser}
 T^{\prime *n}&P_{E_m}T^n e_x = \sum_{y\in\varphi^{-1}(\varphi^{m+1}(x))\setminus\{\varphi^{(m)}(x)\}}  T^{\prime *n}P_{\lin\{e_y\}} T^n e_x 
 \\=& \left\{ \begin{array}{ll}
\hspace{-0.1cm}h(0)^{l}\frac{s(\varphi^{(m)}(x))-|w(\varphi^{(m)}(x))|^2}{s(\varphi^{(m)}(x))}h(m+1)e_x& \textrm{if $n=(l+1)\tau-m$, $l\in \natu$},
\\
   \hspace{-0.1cm}0  & \textrm{otherwise.}
\end{array} \right.\notag
\end{align}
Summing over all $n\in \cal_+$, we get
\begin{align}\label{dec1}
  \sum_{n=1}^\infty  T^{\prime *n}P_{E_m}T^n e_x &=\frac{s(\varphi^{(m)}(x))-|w(\varphi^{(m)}(x))|^2}{s(\varphi^{(m)}(x))}h(m+1) \sum_{l=0}^\infty h(0)^{l}e_x
  \\&=\frac{h(m+1)-h(m)}{1-h(0)}e_x.\notag
\end{align}
Looking at the formula \eqref{eee},  we  deduce  that
\begin{equation}\label{znik}
T^{\prime*} e = \left\{ \begin{array}{ll}
 \overline{ w^\prime(x)}e_{\varphi(x)} & \textrm{if $e=e_x$,  $x\in\gen_\varphi(1,1)$}\\
0 & \textrm{if $e\in\bigoplus_{x\in\gen_\varphi(1,1)}\nul (\com|_{\ell^2(\des(x))})^*)$.}
\end{array} \right.
\end{equation}
Since
$\bigvee\{e_x\colon x\in\gen_\varphi(1,1)\}=\bigoplus_{k=0}^{\tau-1}E_k$ by \eqref{znik} we get
\begin{align}\label{dec2get}
      T^{\prime *n}\pe T^n e_x =  \sum_{k=0}^{\tau-1}  T^{\prime *n}P_{E_k}T^n e_x
\end{align}
which yields
\begin{align}\label{dec2}
    \sum_{n=1}^\infty   T^{\prime *n}\pe T^n e_x =  \sum_{k=0}^{\tau-1} \sum_{n=1}^\infty   T^{\prime *n}P_{E_k}T^n e_x
\end{align}
Combining \eqref{dec1} with \eqref{dec2}, we deduce that (obsereve that $h(0)<1$)
\begin{align*}
\sum_{n=0}^\infty   T^{\prime *n}\pe T^n e_x &=\frac{1}{1-h(0)}\sum_{k=0}^{\tau-1}[h(k+1)-h(k)]e_x = e_x.
\end{align*}
This proves our claim.

Our next goal is to prove that
\begin{equation}\label{sumf}
        \sum_{n=1}^\infty T^{\prime *n}\pe T^nU^*f + \sum_{n=0}^\infty T^n\pe T^{\prime *n}U^*f  = U^*f
    \end{equation}
for $f\in \chil\cap\hol(\ann(r^-,\infty),E)$. Let  $f\in \chil\cap\hol(\ann(r^-,\infty),E)$ and $\{a_x\}_{x\in X}\subset \comp$ be such that 
\begin{equation*}
    U^*f=\sum_{x\in X}a_xe_x\in\ell^2(X).
\end{equation*} We define $f_1, f_2\in\chil$ by
\begin{equation*}
    f_1 := U\big(\sum_{x\in X\setminus Y}a_xe_x\big)\quad \text {and}\quad f_2:=U\big( \sum_{x\in  Y}a_xe_x\big).
\end{equation*}
It follows from
Lemmas \ref{podst}, \ref{cdcom} and \ref{wond} that $\pe T^{\prime*n}e_x=0$, $x\in Y$, $n\in\natu$. This implies 
\begin{align}\label{eqwsp}
    \hat{f_1}(n)=\pe T^{*\prime n}\big(\sum_{x\in X\setminus Y}a_xe_x\big)=\pe T^{*\prime n}\big(\sum_{x\in X}a_xe_x\big)=   \hat{f}(n),\quad n\in \natu.
\end{align}
By kernel-range decomposition and \eqref{eee}, we get
\begin{equation*}
    \hat{f_1}(n)=\pe T^{-n}\big(\sum_{x\in X\setminus Y}a_xe_x\big) =P_{\bigvee\{e_x\colon x\in\gen_\varphi(1,1)\}}T^{-n}\big(\sum_{x\in X\setminus Y}a_xe_x\big)=0
\end{equation*}
for $n\in \cal\setminus\natu$, which yields
\begin{equation*}
    f_1=  \sum_{n=0}^\infty \hat{f}(n)z^n.
\end{equation*}
Combined \eqref{eqwsp} with the fact that  $f\in \hol(\ann(r^-,\infty),E)$ we obtain
\begin{equation*}
    \sum_{n=0}^\infty \hat{f}(n)z^n\in \hol(\comp,E)
\end{equation*} and
\begin{align}\label{rotest}
    \limsup_{n\to\infty}\sqrt[n]{\|\hat{f}(n)\|}=0.
\end{align}

Now we show that  $\sum_{n=0}^\infty T^n\hat{f}(n)$ converges absolutely.
Indeed, applying the root test \cite[page 199]{rud} and \eqref{rotest} we get
\begin{align*}
   \limsup_{n\to\infty}\sqrt[n]{\|T^n\hat{f}(n)\|}\Le  \|T\| \limsup_{n\to\infty}\sqrt[n]{\|\hat{f}(n)\|}=0.
\end{align*}
This  and \eqref{prep} in turn implies that the following  double sum converges




\begin{equation*}
    \sum_{m=0}^\infty\sum_{n=0}^\infty \|T^n\pe T^{\prime *n} T^m\hat{f}(m)\|= \sum_{n=0}^\infty \|T^n\hat{f}(n)\|< \infty.
\end{equation*}
By \eqref{prep} again and changing the order of summation  we have the following equalities
\begin{align}\label{sut1}
    \sum_{n=0}^\infty T^n\pe T^{\prime*n} U^*f_1&=  \sum_{n=0}^\infty T^n\pe T^{\prime *n}\big(\sum_{m=0}^\infty T^m\hat{f}(m)\big)\\&= \sum_{n=0}^\infty \sum_{m=0}^\infty T^n\pe T^{\prime *n}\notag T^m\hat{f}(m)\\&=\sum_{m=0}^\infty\sum_{n=0}^\infty T^n\pe T^{\prime *n}\notag T^m\hat{f}(m)\\&=\sum_{m=0}^\infty T^m\hat{f}(m)=U^*f_1.\notag
\end{align}
By the same
kind of reasoning we see that
\begin{align}\label{sut2}
     \sum_{n=0}^\infty T^{\prime *n} \pe T^n U^*f_1=0.
\end{align}
If $\varphi$ does not have a cycle, then by \eqref{dnhc} we have
\begin{align}\label{doub1}
    \sum_{n=0}^\infty \sum_{x\in  Y} \|T^{\prime *n}\pe T^n a_xe_x\|=\sum_{x\in  Y} a^2_x<\infty.
\end{align}
 Let us pass to the other case when $\varphi$ has a cycle. Since in this case $Y$ is finite, by \eqref{geoser}, \eqref{dec1} and \eqref{dec2get} the following double series converges
 \begin{align}\label{doub2}
      \sum_{n=0}^\infty \sum_{x\in  Y} \|T^{\prime *n}\pe T^n a_xe_x\|&=  \sum_{n=0}^\infty \sum_{x\in  Y} \|\sum_{k=0}^{\tau-1}  T^{\prime *n}P_{E_k}T^n a_xe_x\|
      \\&\Le\sum_{n=0}^\infty \sum_{x\in  Y} \sum_{k=0}^{\tau-1}\| T^{\prime *n}P_{E_k}T^n a_xe_x\|\notag
     \\& =  \sum_{x\in  Y} \sum_{k=0}^{\tau-1} \sum_{n=0}^\infty\| T^{\prime *n}P_{E_k}T^n a_xe_x\|\notag \\&=\sum_{x\in  Y} \sum_{k=0}^{\tau-1}|a_x|\frac{h(k+1)-h(k)}{1-h(0)}<\infty.\notag
 \end{align}
 Using \eqref{odt}, \eqref{doub1}, \eqref{doub2}  and changing the order of summation we get
\begin{align}\label{sut3}
    \sum_{n=0}^\infty T^{\prime *n}\pe T^nU^*f_2&=  \sum_{n=0}^\infty T^{\prime *n}\pe T^n\big(\sum_{x\in  Y}  a_xe_x\big)= \sum_{n=0}^\infty \sum_{x\in  Y}a_xT^{\prime *n}\pe T^n e_x\\&=\sum_{x\in  Y}\sum_{n=0}^\infty a_xT^{\prime *n}\pe T^n e_x=\sum_{x\in  Y} a_xe_x =U^*f_2\notag
\end{align}
Similarly we see that
\begin{align*}
     \sum_{n=0}^\infty T^n \pe T^{\prime *n} U^*f_2=0.
\end{align*}
This combined with \eqref{sut1},  \eqref{sut2} and \eqref{sut3} gives \eqref{sumf}.
An application of Theorem \ref{charak} completes the proof.
\end{proof}
Now we give an example of left-invertible composition operator satisfying the conditions of Theorem \ref{onden}.
\begin{ex}Set $m\in \natu$, $\lambda,\lambda_1,\lambda_2,\dots,\lambda_m\in(0,1)$ and $X=\set{0,1,\dots m}\sqcup \set{(0,i)\colon i\in \natu}$. Let $\varphi:X\to X$ be transformation of $X$ defined by
\begin{equation*}
\varphi(x )= \left\{ \begin{array}{ll}
(0,i-1) & \textrm{for $x=(0,i)$,  $i\in\natu\setminus\set{0}$},\\
m & \textrm{for $x=(0,0)$},\\
i-1 & \textrm{for $x=i$ and $i\in \{1,\dots,m\}$,}\\m&\textrm{for $x=0$,}
\end{array} \right.
\end{equation*}
(see Figure \ref{fig:M3}) and $w:X\to \comp$ be a  function  defined by
\begin{equation*}
w(x )= \left\{ \begin{array}{ll}
1 & \textrm{for $x=0$,  }\\
\lambda_i & \textrm{for $x=i$, $i\in \{1,\dots,m\}$,}
\\
\sqrt{\frac{1}{\lambda^{m+1}}\prod_{i=1}^m\lambda_i} & \textrm{for $x=(0,0)$,  }
\\
\frac{1}{\lambda} & \textrm{for $x=(0,i)$,  $i\in\cal_+$.}
\end{array} \right.
\end{equation*}
Then $\com:\ell^2(X)\to\ell^2(X)$ is a left-invertible composition operator. It is easily seen that
\begin{equation*}\com e_x= \left\{ \begin{array}{ll}
w({(0,i+1)})e_{(0,i+1)} & \textrm{for $x=(0,i)$,  $i\in\natu\setminus\set{0}$,}\\
w(i+1)e_{i+1} & \textrm{for $x=i$ and $i\in\{0,1,\dots,m\}$,}\\
w(0)e_0+w({(0,0)})e_{(0,0)}&\textrm{for $x=m$.}
\end{array} \right.
\end{equation*}
\begin{figure}
\centering
\begin{tikzpicture}
 \begin{scope}[every node/.style={fill=gray!20,circle,thick, minimum width=30pt}]
    \node (A) at (2,2) {$x_{1}$};
    \node (B) at (3.41,1.41) {$x_0$};
    \node (C) at (4,0) {$x_m$};
    \node (D) at (3.41,-1.41) {$x_{m-1}$};
    \node (E) at (2,-2) {$x_{m-2}$};
    \node (F) at (0.59, -1.41) {$x_{m-3}$};
    \node (G) at (5.5,0) {$x_{0,0}$};
    \node (H) at (7,0) {$x_{0,1}$};
    \node[fill=white] (I) at (9.5,0) {};
\end{scope}
\begin{scope}[>={Stealth[black]},
              every node/.style={fill=white,circle},
              every edge/.style={draw=black,very thick}]
    \path [->] (A) edge[bend left=20] (B);
    \path [->] (B) edge[bend left=20]  (C);
    \path [->] (C) edge[bend left=20]  (D);
    \path [->] (D) edge[bend left=20]  (E);
    \path [->] (E) edge[bend left=20]  (F);
    \path [->] (F) edge[bend left=60,loosely dotted]  (A);
    \path [->] (G) edge  (C);
    \path [->] (H) edge  (G);
    \path [->] (I) edge[loosely dotted]  (H);
\end{scope}
\end{tikzpicture}
\caption{} \label{fig:M3}
\end{figure}It is routine to verify that  $\nul \com^*)=\lin\{ \overline{w({(0,0)})}e_{0}-\overline{w({0})}e_{(0,0)}\}$. Let $E:=\lin\{ e_{(0,0)}\}$. One can check that this one-dimensional subspace satisfies \eqref{li}.
The formulas for the inner and outer radius of convergence  take the following form
 \begin{equation*}
 r^+ =\liminf_{n \to \infty}\sqrt[n]{\prod_{i=1}^n|w({(0,i)})|}
\end{equation*}
and
\begin{equation*}
r^-=\sqrt[m+1]{ \prod_{i={0}}^m|w(i)|}
\end{equation*}
(see \cite[Example 5.3]{ja3}). Therefore the inner and outer radius of convergence of \eqref{mod} with $\com$ in place of $T$ are
\begin{equation*}
    r^-=\sqrt[m+1]{\prod_{i=1}^m\lambda_i}\quad\text{and}\quad r^+=\frac{1}{\lambda}.
\end{equation*}
Using Lemma \ref{cdcom} we see that
\begin{equation*}
w^\prime(x )= \left\{ \begin{array}{ll}
1 & \textrm{for $x=0$,  }\\
\frac{1}{\lambda_i} & \textrm{for $x=i$ and $i\in \{1,\dots,m\}$, }
\\
\frac{\lambda^{m+1}}{ \lambda^{m+1}+ \prod_{i=1}^m\lambda_i} & \textrm{for $x=(0,0)$,}
\\
{\lambda} & \textrm{for $x=(0,i)$,  $i\in\natu\setminus\{0\}$.}
\end{array} \right.
\end{equation*}
As a consequence, we obtain the inner and outer radius of convergence  of \eqref{mod} with  $\cdcom$ in place of $T$ 
\begin{align*}
  r^{\prime-}=\lambda\sqrt[m+1]{\frac{\prod_{i=1}^m\lambda_i}{\lambda^{m+1} + \prod_{i=1}^m\lambda_i}} \quad\text{and}\quad r^{\prime+}=\lambda.
\end{align*}
Note that $r^-<1<r^+$ and $r^{\prime-}<r^{\prime+}<1$ therefore $\com$ satisfy assumption of Theorem \ref{onden}.


\end{ex}

\textbf{Acknowledgements.}
The author is  grateful to Professor Jan Stochel for his continual
support and encouragement.
 \bibliographystyle{amsalpha}

\begin{thebibliography}{99}
   
\bibitem{pred} A. Aleman, M. Carlsson, A. M. Persson, Preduals of $Q_p$-spaces, {\em Complex Var. Elliptic Equ.} {\bf 52} (2007), 605--628.
\bibitem{ale} A. Aleman, S. Richter, W. T. Ross,   Pseudocontinuations and the backward shift, {\em Indiana U. Math. J.} {\bf 47} (1998), 223--276.
\bibitem{ana}  A. Anand, S. Chavan, Z. J. Jab{\l}o{\'n}ski,  J. Stochel, A solution to the Cauchy dual subnormality problem for $2$-isometries, {\em J.
Funct. Anal.} {\bf 277} (2019), 108292.

\bibitem{ana2} A. Anand, S. Chavan, Z. J. Jab{\l}o{\'n}ski,  J. Stochel,  The Cauchy dual subnormality problem for cyclic $2$-isometries, {\em Advances in Operator Theory} {\bf 5} (2020), 1061--1077.

\bibitem{ati} M. Atiyah,   Duality in mathematics and physics, In Conferències FME. Vol. V, Curs
Riemann, 2007–2008, Barcelona: Facultat de Matemàtiques i Estadística, pp. 69–91.

\bibitem{badea} C. Badea,  L. Suciu, The Cauchy dual and $2$-isometric liftings of concave operators, {\em J. Math. Anal. Appl} {\bf 472} (2019), 1458--1474.


 




  
  \bibitem{bdp} P. Budzy{\'n}ski,  P. Dymek,  M. Ptak, Analytic structure of weighted shifts on directed trees,  \textit{Mathematische Nachrichten} \textbf{290} (2016), 1612--1629.
  
   \bibitem{planeta} P. Budzy{\'n}ski, P. Dymek, A. P{\l}aneta, M. Ptak,  Weighted shifts on directed trees. Their multiplier algebras, reflexivity and decompositions, \textit{Studia Math.} \textbf{244} (2019), 285--308.

 
  
\bibitem{2hyp}  S. Chavan, On operators Cauchy dual to 2-hyperexpansive operators, {\em Proc. Edinburgh Math. Soc.} {\bf 50} (2007), 637--652.

\bibitem{un2hyp}  S. Chavan, On operators Cauchy dual to 2-hyperexpansive operators: the unbounded case, {\em Studia Math.} {\bf 2} (2011), 129--162.

\bibitem{curto} S. Chavan, R. Curto, Operators Cauchy dual to 2-hyperexpansive operators: the multivariable case, {\em Integr. Equ. Oper. Theory} {\bf 73} (2012), 481--516.







  
 

 





\bibitem{chav} S. Chavan, S. Trivedi, An analytic model for left-invertible weighted shifts on directed trees,
\textit{J. London Math. Soc.} \textbf{94} (2016), 253--279.

\bibitem{Sark} S. Das,  J. Sarkar, Aluthge transforms, Tridiagonal kernels, and left invertible operators, arXiv preprint arXiv:2009.03410 (2020).

\bibitem{dym} P. Dymek, A. P{\l}aneta,  M. Ptak, Generalized multipliers for left-invertible analytic operators and their applications to commutant and reflexivity, \textit{J. Funct. Anal.} \textbf{276}  (2019), 1244--1275.
 
 \bibitem{ezz} H. Ezzahraoui, M. Mbekhta,  E. H. Zerouali. On the Cauchy dual of closed range operators,  {\em Acta Sci. Math.(Szeged)} {\bf 85} (2019), 231--248.









\bibitem{garr} P. Garrett, Modern Analysis of Automorphic Forms By Example, vol. II of Cambridge
Studies in Advanced Mathematics, Cambridge University Press, 2018.

\bibitem{goha} I. Gohberg, J. Leiterer,  Holomorphic Operator Functions of One Variable and
Applications Operator Theory: Advances and Applications, 192. Birkh\"auser.




\bibitem{Muj} J. Mujica, Complex analysis in Banach spaces, North-Holland Math. Studies, vol. 120, North Holland, Amsterdam, 1986.



\bibitem{ja3} P. Pietrzycki, A Shimorin-type analytic model on an annulus for left-invertible
operators and applications, {\em J. Math. Anal. Appl.} {\bf 477} (2019), 885--911.

\bibitem{ja4} { P. Pietrzycki},  Generalized multipliers for left-invertible operators and applications, {\em Integr. Equ. Oper. Theory} {\bf 92} (2020), 41

\bibitem{PoPr21} S. Podder, D.  K. Pradhan, The reflexivity of hyperexpansions and their Cauchy dual operators, {\em Oper. Matrices} {\bf 15} (2021), 195--207.

\bibitem{shap} W.T. Ross,  H. S. Shapiro,  Generalized analytic continuation, University Lecture Series,
 (Providence: American Mathematical Society, RI) Vol. 25 (2002), pp. xiv149.


\bibitem{rud} W. Rudin, Functional Analysis, McGraw-Hill, International Editions 1991








\bibitem{shi}S. Shimorin,  Wold-type decompositions and wandering subspaces for operators close to isometries,\textit{ J. Reine Angew. Math.} \textbf{531} (2001), 147--189.





   \end{thebibliography}

\end{document}